\documentclass[a4paper,11pt]{amsart}
\usepackage{amsfonts}
\usepackage{mathrsfs}
 \usepackage[arrow,matrix]{xy}
\usepackage{amsmath,amssymb,amscd,bbm,amsthm,mathrsfs,dsfont}

\newtheorem{thm}{Theorem}[section]

\newtheorem{cor}{Corollary}[section]
\newtheorem{prop}{Proposition}[section]
\newtheorem{rem}{Remark}[section]

\theoremstyle{definition}

\newcommand{\bH}{{\mathbb{H}}}

\newcommand{\bfi}{{\mathbf{i}}}
\newcommand{\bfj}{{\mathbf{j}}}
\newcommand{\bfk}{{\mathbf{k}}}

\DeclareMathOperator{\re}{Re}
\DeclareMathOperator{\im}{Im }

\begin{document}
\numberwithin{equation}{section}

 \title[On Cauchy-Szeg\"o kernel for quaternionic Siegel upper half space ]
 {On Cauchy-Szeg\"o kernel  for quaternionic Siegel upper half space  }
\author {Der-Chen Chang, Irina Markina  and Wei Wang
 }

\keywords{The quaternionic  regular  functions, Siegel upper half space, Cauchy - Szeg\"o kernel, projection operator, Hardy space}

\subjclass[2000]{43A15, 42B35, 42B30}

\address{Department of Mathematics and Department of Computer Science,
Georgetown University, Washington D.C. 20057, USA
\newline
Department of Mathematics, Fu Jen Catholic University, Taipei 242, Taiwan, ROC}
\email{chang@georgetown.edu}

\address{Department of Mathematics, University of Bergen, NO-5008 Bergen, Norway}
\email{irina.markina@uib.no}

\address{Department of Mathematics, Zhejiang University, Zhejiang 310028, PR China}
\email{wwang@zju.edu.cn.}

\thanks{The first author is partially supported by an NSF grant DMS-1203845 and Hong Kong
RGC competitive earmarked research grant $\#$601410. The second author gratefully acknowledge partial support by the grants NFR-204726/V30 and NFR-213440/BG, Norwegian Research Council. The third author is partially supported by National Nature Science Foundation in China (No. 11171298).
}

\begin{abstract} The work is dedicated to the construction of the Cauchy-Szeg\"o kernel for the Cauchy-Szeg\"o projection integral operator from the space of $L^2$-integrable functions defined on the boundary of the quaternionic Siegel upper half space  to   the space of boundary values of the quaternionic  regular functions of the Hardy space over the quaternionic Siegel upper half space.
\end{abstract}

\maketitle

\section{\bf Introduction}

It is a well known fact that the unit disc (or 2 dimensional ball) is bi-holomorphically equivalent to the upper half space of the complex plane by Cayley transform. The abelian group $(\mathbb R,+)$ acts as translations parallel to the boundary in the upper half plane and can be extended to the boundary. Since the action of the group $(\mathbb R,+)$ is transitive on the boundary, the boundary can be identified with the group by its action on the origin. Passing to the two dimensional complex plane one obtains that 4 dimensional real open ball is bi-holomorphically equivalent to the Siegel upper half space by 2 dimensional Cayley transform.
The abelian group $(\mathbb R,+)$ is replaced by non-abelian Heisenberg group, that is a subgroup of the group of automorphisms of the
Siegel upper half space and also can be extended to the transitive action on the boundary. It allows us to identify the points on the boundary of the
Siegel upper half space with the Heisenberg group. This construction can be generalized to the $n$-dimensional complex space.
Moreover, if we change 2 dimensional complex space by 2 dimensional quaternionic space, then the corresponding Cayley transform maps $8$ dimensional real open ball to the quaternionic Siegel upper half space and it extends to the boundary. The analogue of the Heisenberg group is the, so called,
quaternionic Heisenberg group and it forms a subgroup of the group of automorphisms of the quaternionic Siegel upper half space.
Extending the action of the quaternionic Heisenberg group to the boundary of the Siegel upper half space and taking into account its transitive action,
one realizes the boundary as a group. As in the case of complex variables, the latter situation can be generalized to the multidimensional quaternionic space.

The classical Hardy space $H^2(\mathbb R^2_+)$ consists of holomorphic functions on the upper half plane $\mathbb R^2_+$ such that
$$
\sup_{y>0}\int_{-\infty}^{+\infty} \big|f(x+iy)\big|^2 dy\,<\, +\infty.
$$
Standard argument shows that such functions have boundary values in $L^2(\mathbb R)$. (See {\it e.g.} Chapter 3 in \cite{Stein-Weiss} and Chapter 2 in \cite{St1}). The set of all boundary values forms a closed subspace of
$L^2(\mathbb R)$ and the Cauchy-Szeg\"o integral is the projection operator from $L^2(\mathbb R)$ to this closed subspace. The Cauchy-Szeg\"o integral is
written as a convolution with the Cauchy-Szeg\"o kernel, that in the same time is the reproducing kernel for the functions from the
Hardy space $H^2(\mathbb R^2_+)$. Following this line, in the books~\cite{St1} and ~\cite{St} the construction of the Cauchy-Szeg\"o kernel was realized as a kernel
of the projection operator from $L^2(\partial\mathcal U_n)$ space of functions on the boundary $\partial\mathcal U_n$ of the Siegel upper half space to the space of
boundary values of the functions from Hardy space $H^2(\mathcal U_n)$ over the Siegel upper half space.
The projection operator is given as a convolution with respect to the Heisenberg group product and has the reproducing property. In the present paper, we present analogue of this construction for the quaternionic  regular  functions, the quaternionic Siegel upper half space and the quaternionic Heisenberg group. We compute the Cauchy-Szeg\"o kernel explicitly for any dimension $n$. The construction is much more complicated than in the case $\mathbb C^2$.

We denote by $\mathbb H$ the space of quaternionic numbers $q=x_1+x_2 \mathbf{i}+x_3 \mathbf{j}+x_4\mathbf{k}$.
We write $\re\bH$ for one dimensional subspace of $\bH$ spanned by $1$  and $\im\bH$ for 3 dimensional subspace of $\bH$ spanned by $\{\bfi,\bfj,\bfk\}$.
The $n$-dimensional quaternionic space $\bH^n$ is the collection of $n$-tuples $(q_1,\ldots,q_n)$, $q_l\in\bH$.
For $l$-th coordinate of a point $q=(q_1,\cdots,q_{n   })\in \mathbb H^{n }$ we write
\begin{equation}\label{eq:ql}
 q_{l  }=x_{4 l-3}+x_{4 l-2}\mathbf{i}+x_{4 l-1}\mathbf{j}+x_{4 l  }\mathbf{k},\qquad l=1,\ldots,n.
 \end{equation}For a domain $D\subset \bH^n$, a $C^1$-smooth function $f=f_1+
 \mathbf{i}f_2
 + \mathbf{j}f_3+ \mathbf{k}f_4\colon  D\rightarrow \mathbb H$ is called {\it $($left $)$ regular on $D$} if
it satisfies the Cauchy-Fueter equations
\begin{equation}\label{eq:CF}
     \overline{\partial}_{ q_{l  }}   f(q)=0,\qquad l=1,\ldots,n ,\quad q\in D,
\end{equation}
where
\begin{equation}\label{eq:Cauchy-Fueter}
   \overline\partial_{{q}_{l  }}=\partial_{x_{4 l-3}}+
 \mathbf{i}\partial_{x_{4 l-  2}}
 + \mathbf{j}\partial_{x_{4 l-1 }}+ \mathbf{k}\partial_{x_{4 l }}.
\end{equation}
Recently, people are interested in developing a theory for the   regular  functions of several quaternionic variables, as the counterpart of the theory of several complex variables for holomorphic functions (see~ \cite{adams2}, ~ \cite{ALPS},  ~\cite{AL},~ \cite{bures}, ~\cite{colombo}, ~\cite{colombo2006}, ~\cite{Wa08}, ~\cite{Wang10}, ~\cite{Wang11} and references therein).

The {\it quaternionic Siegel upper half space} is
\begin{equation}\begin{split} \mathcal{U}_n:=\left\{q=(q_1,\ldots,q_n)=(q_1,q')\in \mathbb{H}^{n }\vert\ \
\re q_{1  }> |q'|^2\right\},
\end{split}\end{equation}
where we denoted $q'=(q_2,\ldots,q_n)\in \mathbb{H}^{n-1 }$. Its boundary $\partial\mathcal{U}_n$ is a quadratic hypersurface defined by equation
\begin{equation}
\re q_{1  }= |q'|^2 .
 \end{equation}

Notice that the quaternionic space $\bH^n$ is isomorphic to $\mathbb R^{4n}$ as a vector space and the pure imaginary quaternions $\im\bH$ are isomorphic to $\mathbb R^3$. The {\it quaternion Heisenberg group} $qH^{n-1}$ is the space $\mathbb R^{4n-1}=\mathbb R^3\times \mathbb R^{4(n-1)}$, that is isomorphic to $\im\bH\times\bH^{n-1}$, furnished with the non-commutative product
\begin{equation} \label{eq:multiple-H}
p\cdot q=(w,p')\cdot (v,q')=\left(w+v+2\im \langle {p}',q'\rangle,p'+q'
 \right),
\end{equation}
where $p=(w,p'), q=(v,q')\in\im\bH\times\bH^{n-1}$, and $\langle\cdot,\cdot\rangle$ is the inner product defined in (\ref{eq:inner-product}) on $\bH^{n-1}$.

The {\it projection}
\begin{equation}\label{eq:proj}\begin{split}
\pi\colon \partial\mathcal{U}_n&\longrightarrow \im\mathbb{H} \times\mathbb{H}^{n-1},\\
  (|q'|^2+ x_2\mathbf{i}+x_3\mathbf{j}+x_4\mathbf{k},q')  &\longmapsto (x_2\mathbf{i}+x_3\mathbf{j}+x_4\mathbf{k},q').
\end{split}\end{equation}
identifies the boundary of the quaternionic Siegel upper half space $\partial\mathcal U_n=\{(q_1,q')\in \mathcal U_n\vert\ \ \re q_1=|q'|^2\}$ with the quaternionic Heisenberg group $qH^{n-1}$. Let $d\beta (\cdot)$ be   the Lebesgue measure on $\partial \mathcal U_n$ obtained by pulling back by the  projection $\pi$~\eqref{eq:proj} the Haar measure on the group $qH^{n-1}$.

For any function $F\colon \mathcal{U}_n\to\mathbb H$, we write $ F_\varepsilon$  for its "vertical translate". We mean that the
vertical direction is given by the positive direction of $\re q_1$:
 $ F_\varepsilon(q) = F(q + \varepsilon \mathbf{e})$, where $\mathbf{e} = (1,0, 0,\ldots, 0)$.
If  $\varepsilon> 0$, then  $ F_\varepsilon$  is defined in a neighborhood of $\partial\mathcal{U}_n$. In particular, $ F_\varepsilon$  is defined on  $\partial\mathcal{U}_n$.
The {\it  Hardy space}
  $H^2(\mathcal{U}_n )$ consists of all regular  functions $ F$  on $\mathcal{U}_n$, for which
 \begin{equation}\label{eq:def-H2}
     \sup_{\varepsilon>0}\int_{\partial\mathcal{U}_n }|F_\varepsilon(q)|^2 d\beta (q) <\infty.
 \end{equation}
The norm $\|F\|_{H^2(\mathcal{U}_n )}$ of $F $ is then the square root of the left-hand side of (\ref{eq:def-H2}).
A function $F\in H^2(\mathcal{U}_n )$ has boundary value $F^b$ that belongs to $L^2(\partial\mathcal U)$ by Theorem
\ref{thm:boundary_value}.

Now we can state the main result of the paper.
\begin{thm}\label{thm:CS}
The Cauchy-Szeg\"o kernel is given by
\begin{equation}\label{eq:CS}
S(q, p) =s\left(q_1+\overline{p}_1  -2\sum_{k=2}^{n} \bar p_k'q'_k \right),
     \end{equation}
 for $p=(p_1,p')=(p_1,\ldots, p_n)\in \mathcal{U}_n $, $q=(q_1,q')=(q_1,\ldots, q_n) \in \mathcal{U}_n $, where
     \begin{equation}\label{eq:s}
         s( \sigma )= c_n  \frac {\partial^{2n}}{\partial x_1^{2n}}\frac {\overline{\sigma} }{|\sigma|^4},\qquad \sigma=x_1+x_2 \mathbf{i}+x_3 \mathbf{j}+x_4\mathbf{k} \in \mathbb H.
     \end{equation}
 Here
\begin{equation}
\label{eq:constant}
c_n=\frac{1}{2^{2n+5}\pi^{2n+1}\big((2n)!\big)^2K(n)}\frac{4n-1}{(n+2)(2n+3)},
\end{equation}
where the constant
\begin{align}\label{K}
K(n)& =\sum_{k=0}^{2n}\alpha_k\sum_{l=0}^{k}C^{2l}_{k}\sum_{m=0}^{l}(-1)^{k+m}C^{m}_{l}\nonumber
\\
&\sum_{s=0}^{k-2m}\frac{C_{k-2m}^s}{2^{k-2m-s+1}}\frac{(-1)^{s}\big(2(k-2m-s+1)\big)!}{(k-2m-s+1)!(4n+5+k-2m-s)!}
\end{align}
depends only on the dimension $n$, and $\alpha_k=\frac{(2n+1-k)(2n+2-k)(4n+3+k)}{6}$.

The Cauchy-Szeg\"o kernel satisfies the reproducing property in the following sense
 \begin{equation}\label{eq:Szego}
    F(q) =
\int_{\partial\mathcal{U}_n}
S(q, Q)F^b(Q) d\beta (Q),  \qquad q\in \mathcal{U}_n,
 \end{equation}
whenever $F \in   H^2(\mathcal{U}_n)$ and $F^b$ its boundary value on $\partial U_n$.
\end{thm}

The paper is organized as follows. In Section~\ref{sec:quaternion} we recall the structure of quaternion numbers and the Siegel upper half space,
mentioning some invariance properties. In Section~\ref{sec:regular} we study regular functions in domains of multidimensional quaternionic space.
In Section~\ref{sec:Hardy} we discuss the boundary value of regular functions in  the Siegel upper half space $\mathcal{U}_n$ and
invariance properties of the Hardy space $H^2(\mathcal{U}_n)$ over
$\mathcal{U}_n$. The main part of Section~\ref{sec:kernel} is devoted to determining the Cauchy-Szeg\"o kernel $S$ and the proof of Theorem~\ref{thm:CS}.


\section{\bf The quaternionic Siegel upper half space}\label{sec:quaternion}


\subsection{\bf Right quaternionic vector space}


   The space $\mathbb H$ of quaternionic numbers  forms a division algebra with respect to the coordinate addition and the quaternion multiplication
\begin{equation}\label{eq:q-sigma0}
\begin{split}
     q\sigma=& (x_1+\mathbf{i}x_2+\mathbf{j}x_3+\mathbf{k}x_4)( \sigma_1+ \mathbf{i}\sigma_{  2 }+ \mathbf{j}\sigma_{ 3 }+ \mathbf{k}\sigma_{ 4 })
     \\=&
     \hskip 6mm \sigma_{ 1 }x_1-  \sigma_{ 2 }x_2-  \sigma_{ 3 }x_3- \sigma_{ 4 }x_4
     +(\sigma_{ 2 }x_1+ \sigma_{ 1 } x_2+\sigma_{ 4 }x_3   -\sigma_{ 3 }x_4  )\mathbf{i}
     \\  + &
     \ \ ( \sigma_{ 3 } x_1-\sigma_{ 4 }x_2+ \sigma_{ 1 }x_3+ \sigma_{ 2 }x_4  )\mathbf{j}
     +(\sigma_{ 4 } x_1+  \sigma_{ 3 } x_2  -\sigma_{ 2 }x_3+\sigma_{ 1 }x_4) \mathbf{k},
\end{split}
\end{equation}
for $q,\sigma\in\bH$. Denote by $\re q=x_1$ the real part of $q$ and by $\im q$ the imaginary part of $q$ that is a three dimensional vector $\overrightarrow r=(x_2,x_3,x_4)$.

The conjugate $\overline{q}$ of a quaternion $q=x_1+x_2 \mathbf{i}+x_3 \mathbf{j}+x_4\mathbf{k}$ is defined by $  \overline{q}=x_1-x_2\mathbf{i}-x_3\mathbf{j}-x_4\mathbf{k}$ and the norm is $|q|^2=\overline q q$. The conjugation inverses the product of quaternion numbers in the following sense
$\overline{\sigma q}=\overline q\cdot\overline\sigma$ for any $\sigma,\tau \in \mathbb H$.
As a vector space, $\bH$ is isomorphic to $\mathbb R^4$.

Since the quaternionic algebra $\mathbb{H}$ is associative, although it is not commutative, there is a natural notion of a vector space over $\mathbb{H}$, and many
definitions and propositions for real or complex linear algebra also hold for quaternionic linear spaces, see~\cite{A03}, ~\cite{Port} ,~\cite{Wa11}.
Let us recall here some definitions and basic properties of vector spaces over $\bH$.

A {\it right quaternionic vector space} is a set $V$ with addition $+\colon V\times V\rightarrow V$ and {\it right scalar multiplication} $ V\times \mathbb{H} \rightarrow V, (v,\sigma)\mapsto v \sigma$, where $V$ is an abelian group with respect to the addition, and the right scalar multiplication satisfies the following axioms:
\begin{itemize}
\item[(1)]{ $(v+w)\sigma=v\sigma+w \sigma$,}
\item[(2)]{$ v(\sigma_1+\sigma_2)=v\sigma_1+v \sigma_2$,}
\item[(3)]{ $ v(\sigma_1 \sigma_2)=(v\sigma_1)\sigma_2$,}
\item[(4)]{ $v1=v$,}
\end{itemize}
for any $v,w\in V$ and $\sigma, \sigma_1,\sigma_2\in \mathbb{H}$.

A {\it  hyperhermitian semilinear form} on a right quaternionic vector space  $V$  is a map
$a\colon V \times V \longrightarrow \mathbb{H}$
satisfying the following properties:
\begin{itemize}
\item[(1)]{ $a$ is additive with respect to each argument,}
\item[(2)]{ $a(q, q'   \sigma)= a(q, q' )   \sigma$ for any $q,q'  \in V$ and any $\sigma\in \mathbb{H} $,}
\item[(3)]{ $a(q, q')= \overline{a(q', q)}$.}
\end{itemize}
Properties
(2) and (3) imply that $a$ is conjugate right   linear with respect to the first argument:  $ a(q\sigma, q'   )= \overline{\sigma} a(q, q' )  $.

A quaternionic $(n\times n)$-matrix $A $ is called {\it hyperhermitian} if $A^*=A$, where $(A^*)_{jk}:=\overline{A}_{kj}$.  For instance, for $ q = (q_1, \ldots ,q_n)$, $p = (p_1,  \ldots ,p_n)\in\mathbb{H}^n$,
set
$a(q, p)=
\sum_{i,j}
 \overline{q}_iA_{ij} p_j$.
  Then $a(\cdot, \cdot)$ defines a hyperhermitian semilinear form on $\mathbb{H}^n$.

A hyperhermitian semilinear form $a(\cdot, \cdot)$ is called {\it positive definite} if
 $a(v,v)\geq 0$ for any $v\in V$, and $a(v,v)= 0$ if and only if $v=0$.
 A  positive definite hyperhermitian semilinear form $ a(\cdot, \cdot)$ on a  right quaternionic vector space is called an {\it  inner product} and will be denoted from now on by $ \langle v, w\rangle:=a(v, w)$.

Now set
\begin{equation}\label{eq:norm}
   \|v\|:=\langle v,v\rangle^{\frac 12},\qquad\text{and}\qquad \rho(v,w)=\|v-w\|.
\end{equation}
The value $\|v\|$ is called the norm of $v\in V$ and $\rho(v ,w)$ is a distance between $v$ and $w$ on $~V$. To show that $\rho(\cdot ,\cdot)$ is a distance, we need the quaternion version of the Cauchy-Schwarz inequality:
$
     |\langle v,w\rangle|\leq \|v\|\|w\|,
$
that follows from the observation
\begin{equation*}
     0\leq \langle v-w\sigma,v-w\sigma\rangle=\langle v ,v\rangle-\overline{\sigma}\langle w,v \rangle-\langle v, w\rangle \sigma+|\sigma|^2\langle w, w \rangle.
\end{equation*}
Write $\langle v ,w \rangle=r\xi$ for a unit quaternion $\xi$ and  $r\geq0$, and choose $\sigma=t\overline{\xi}$, $t\in \mathbb{R}$. Then we find that $0\leq\|v\|^2-2rt+t^2\|w\|^2$ for any $t$. The Cauchy-Schwarz inequality follows. This makes $V$ as a space of homogeneous type.

If  $\rho(\cdot ,\cdot)$ is a complete  distance, we call $(V, \langle \cdot, \cdot\rangle )$ a {\it right quaternionic Hilbert space}.

\begin{prop}{\rm (The quaternion version of Riesz's representation theorem)} Suppose that $(V,\langle \cdot,\cdot\rangle)$ is a right quaternionic Hilbert space and $h\colon V\rightarrow  \mathbb{H}$ is a bounded right quaternionic linear functional: $h$ is additive and $h(v\sigma)=h(v)\sigma$ for any $v \in V$ and $\sigma \in \mathbb{H}$. Then there exists a unique element $v_h\in V$ such that
\begin{equation*}
    h(v)=\langle v_h,v\rangle,\qquad\text{ for any}\quad v\in V.
\end{equation*}
\end{prop}
\begin{proof}
Let $M=\ker h$, where $\ker h$ is the kernel of the linear functional $h$. Then $M$ is a closed subspace because $h$ is continuous. Moreover, $M$ is a right quaternionic linear space since $h$ is.
Set $M^\perp:=\{v\in V\vert\ \  \langle w,v\rangle=\overline{\langle v,w \rangle}=0$ for any $w\in M\}$. If $h$ is non-vanishing, then $M\neq V$ and so $M^\perp\neq\{0\}$. Thus, there exists an element $v_0\in M^\perp$ such that $h(v_0)=1$. Now $h(v-v_0h(v))=h(v)-h(v_0)h(v) =0$ for any $v\in V$, i.e., $v-v_0h(v)\in M$. So
\begin{equation*}
    0=\langle v_0,v-v_0h(v)\rangle =\langle v_0,v\rangle -\|v_0\|^2h(v).
\end{equation*}
Namely, we can choose $v_h=v_0\|v_0\|^{-2}$. The uniqueness is easily follows from the positive definiteness of the product.
\end{proof}

At the end of the subsection we notice that the space $\bH^n$ is a complete right quaternionic Hilbert space endowed with the inner product
\begin{equation}
\label{eq:inner-product}
\langle p,q\rangle=\sum_{l=1}^{n}\bar p_l q_l,\qquad p=(p_1,\ldots,p_n),\ \ q=(q_1,\ldots,q_n)\in\bH^n.
\end{equation}


\subsection{\bf The quaternionic Siegel upper half space and the quaternionic Heisenberg group}


The next step is to present transformations acting on the Siegel upper half space.
A quaternionic $(n\times n)$-matrix $ \mathbf{a}=(a_{jk})$ acts on $\mathbb H^n$ on left as follows:
\begin{equation}\label{eq:A}
     q\longmapsto  \mathbf{a}q,\qquad ( \mathbf{a}q)_j= \sum_{k=1}^n a_{jk}q_k
\end{equation}
for $q=(q_1,\ldots q_n)^t $, where the upper index $^t$ denotes the transposition of the vector.  Note that the transformation in (\ref{eq:A}) commutes with right multiplication by $\mathbf{i}_\beta $ ($\mathbf{i}_1= 1$, $\mathbf{i}_2= \mathbf{i}$, $\mathbf{i}_3= \mathbf{j}$, $\mathbf{i}_4= \mathbf{k}$), i.e.
\begin{equation*}
   ( \mathbf{a}q)\mathbf{i}_\beta = \mathbf{a}(q\mathbf{i}_\beta ) .
\end{equation*}    Namely, $\mathbf{a}$ transforms a right  quaternionic line to a right  quaternionic line.  Here the  {\it right quaternionic line} through the origin and the point $q=(q_1,\ldots,q_n)^t$ we mean the set
$\{(q_1\sigma,\ldots,q_n\sigma)^t|\sigma\in \mathbb{H} \}$. The group ${\rm
GL}(n ,\mathbb H )$ is isomorphic to the group of all  linear  transformations   of $\mathbb{R}^{4n}$ commuting with  $\mathbf{i}_\beta $, while the compact Lie group $ {\rm
Sp}(n  )$ consists of   orthogonal  transformations   of $\mathbb{R}^{4n}$ commuting with    $\mathbf{i}_\beta $.

\begin{prop} \label{prop:transformations} The   Siegel upper half  space $\mathcal{U}_n$ is invariant under the following transformations.
\begin{itemize}
\item[(1)] {Translates:
\begin{equation}\label{eq:translates}
     \tau_p : (q_1, q') \longmapsto \left(q_1+p_1+2\langle p',q'\rangle,q'+p'\right),
\end{equation}
for  $p=(p_1,p')=(p_1,\ldots,p_n)\in \partial\mathcal{U}_n$, where $p'=(p_2,\ldots,p_n)\in \mathbb{H}^{n-1 }$. }
\item[(2)]{ Rotations:
\begin{equation}\label{eq:rotations-1}
  R_{\mathbf{a}} :  (q_1 ,q' )\longmapsto (q_1 ,\mathbf{a}q' )
\end{equation}
for   $\mathbf{a}\in {\rm Sp}(n-1)$, and
\begin{equation}\label{eq:rotations-2}
  R_{\sigma} :   (q_1 ,q' )\longmapsto (\overline{\sigma} q_1 \sigma , q'\sigma )
\end{equation}
for   $\sigma\in \mathbb{H}$ with $|\sigma|=1$.}
\item[(3)]{ dilations:
 \begin{equation}\label{eq:dilations}
 \delta_r:    (q_1 ,q' )\longmapsto ( r^2 q_1  , rq'  ),\quad r>0.
\end{equation}
}
\end{itemize}
All the maps are extended to the boundary $\partial\mathcal U_n$ and transform the boundary $\partial\mathcal U_n$ to itself. Moreover, all the maps transform  the hypersurface   $\partial\mathcal U_n+\varepsilon\mathbf{e}$ to itself for each $\varepsilon>0$.
\end{prop}

\begin{proof} The formula~\eqref{eq:translates} follows from
\begin{equation}\label{eq:quadra-multi}\begin{split}
   &{\rm Re}( q_1+p_1+2\langle p',q'\rangle )-|q'+p'|^2\\=&\re q_1+\re p_1+ 2{\rm Re}\langle p',q'\rangle-(|q'|^2+|p'|^2+2{\rm Re}\langle p',q'\rangle)\\=&\re( q_1 )-|q' |^2  >0
\end{split}\end{equation}
by $\re p_1=|p'|^2$.

The rotations (\ref{eq:rotations-1}) obviously map   $\mathcal{U}_n$ to itself. For rotations (\ref{eq:rotations-2}), note that
\begin{equation}\label{eq:unit}
   q_1^2=-1\hskip 3mm   \text{   if and only if   } \hskip 3mm  x_1=0 \text{   and   }x_2^2+x_3^2+x_4^2=1
\end{equation}
for a quaternion number $q_1=x_1+\mathbf{i}x_2+\mathbf{j}x_3+\mathbf{k}x_4 $.
  This is because of
\begin{equation*}
     q_1^2=x_1^2+2x_1(\mathbf{i}x_2+\mathbf{j}x_3+\mathbf{k}x_4 )+(\mathbf{i}x_2+\mathbf{j}x_3+\mathbf{k}x_4 )^2
\end{equation*}
and
\begin{equation}\label{eq:sum-square}
(\mathbf{i}x_2+\mathbf{j}x_3+\mathbf{k}x_4 )^2=-|\mathbf{i}x_2+\mathbf{j}x_3+\mathbf{k}x_4 |^2=-x_2^2-x_3^2-x_4^2 .
\end{equation}
    Since
\begin{equation}\label{eq:sigma-q}
    \overline{\sigma} q_1 \sigma =x_1+\overline{\sigma}(\mathbf{i}x_2+\mathbf{j}x_3+\mathbf{k}x_4)\sigma,
\end{equation}
 and
\begin{equation*}
        \overline{\sigma}( \mathbf{i}x_2+\mathbf{j}x_3+\mathbf{k}x_4)\sigma
        \overline{\sigma}(\mathbf{i}x_2+\mathbf{j}x_3+\mathbf{k}x_4)\sigma=-x_2^2-x_3^2-x_4^2,
\end{equation*}by (\ref{eq:sum-square}),
we see that the second term in the right hand side of (\ref{eq:sigma-q}) is imaginary by using (\ref{eq:unit}). Consequently, $\re (\overline{\sigma} q_1 \sigma)=\re(  q_1  )$ and so
\begin{equation}\label{eq:adjoint-sigma}
     \re  (\overline{\sigma} q_1 \sigma)-| q'\sigma |^2=\re (  q_1     )-|q'|^2.
\end{equation}
The invariance  of the hypersurface   $\partial\mathcal U_n+\varepsilon\mathbf{e}$ under the maps  $\tau_p$ and $R_\sigma$ follows from (\ref{eq:quadra-multi}) and (\ref{eq:adjoint-sigma}).
 The other statements  are obvious. The result follows.
\end{proof}

The total group of rotations for $\mathcal U_n$ is ${\rm Sp}(n-1){\rm Sp}( 1)$ with ${\rm Sp}( 1)\cong \{\sigma\in \mathbb{H}\vert\   |\sigma|=1 \}$.

\begin{rem}
Translate $\tau_p$ can be viewed as an action of the quaternionic Heisenberg group $qH^{n-1}$ on the quaternionic Siegel upper half space $\mathcal U_n$.  Let $p=(v,p')\in qH^{n-1}$, then the translates~\eqref{eq:translates} can be written as
$$
\tau_p : (q_1, q') \longmapsto (q_1+|p'|^2+v+2\langle p',q'\rangle,q'+p').
$$
It is obviously extended to the boundary $\partial \mathcal U_n$. It is easy to see that the action on $\partial \mathcal U_n$ is transitive, for calculation see also~\cite{CChM}. Therefore, we can identify points in $qH^{n-1}$ with points in $\partial \mathcal U_n$ by the result of the translates by $\tau_p$ of the origin $(0,0)$.
\end{rem}


\section{\bf Regular functions on the quaternionic Siegel upper half space}\label{sec:regular}


In the present Section we   show the invariance of the regularity under linear transformations in Proposition \ref{prop:transformations}.

\begin{prop} \label{prop:D-A}  Let $f \colon D\to\mathbb{H}$ be $C^1$-smooth function, where $D$ is a domain in~$\mathbb H^n$.
\begin{itemize}
\item[(1)]{ Define the pull-back function $\hat f$ of $f$ under the mapping $  q \rightarrow Q= \mathbf{a}   q  $ for $ \mathbf{a}=({ a}_{jk })\in {\rm
GL}( n ,\mathbb H)$ by
$
     \widehat{f}(q ):=f( \mathbf{a} q )$. Then we have
\begin{equation}\label{eq:D-A'}
    \overline{\partial}_{  q_j }  \widehat{f}(   q ) =\sum_{k =1}^n  \left.\overline{ a}_{kj }\overline{\partial}_{Q_k }f( Q )\right|_{Q= \mathbf{a}   q}.
\end{equation}}
\item[(2)] { Define the pull-back function $\widetilde f$ of $f$ under the mapping $ q \rightarrow Q=   q \sigma $ for $\sigma\in \mathbb H$ by
$
     \widetilde{ {f}}(  q ):=f(  q_1 \sigma,\ldots, q_n \sigma )
$.
     Then
\begin{equation}\label{eq:D-sigma}
    \overline{\partial}_{  q_l }\widetilde{ {f}}(  q )=  \left. \overline{\partial}_{Q_l } [ \overline{\sigma} f ( Q)]\right|_{Q=     q\sigma},\qquad l=1,\ldots,n.
\end{equation}
}
\end{itemize}
\end{prop}
\begin{proof}
The proof of the first statement can be found in~\cite[ Proposition 3.1]{Wa11}.

The second statement is analogous to the formula of one quaternionic variable. Write the $l$-th coordinate of $q=(q_1,\ldots,q_n)$ as the quaternionic number $q_l= x_1+\mathbf{i}x_2+\mathbf{j}x_3+\mathbf{k}x_4\in \mathbb{H} $, and define the {\it associated real vector $q^{\mathbb{R}}_l:=(x_1,x_2,x_3,x_4)^t$} in $\mathbb{R}^{4 }$. Then for the product
\begin{equation}\label{eq:q-sigma}\begin{split}
     q_l\sigma=& (x_1+\mathbf{i}x_2+\mathbf{j}x_3+\mathbf{k}x_4)( \sigma_1+ \mathbf{i}\sigma_{  2 }+ \mathbf{j}\sigma_{ 3 }+ \mathbf{k}\sigma_{ 4 })\\=& \hskip 6mm \sigma_{ 1 }x_1-  \sigma_{ 2 }x_2-  \sigma_{ 3 }x_3- \sigma_{ 4 }x_4\\ &+(\sigma_{ 2 }x_1+ \sigma_{ 1 } x_2+\sigma_{ 4 }x_3   -\sigma_{ 3 }x_4  )\mathbf{i}\\&  +( \sigma_{ 3 } x_1-\sigma_{ 4 }x_2+ \sigma_{ 1 }x_3+ \sigma_{ 2 }x_4  )\mathbf{j}\\&+(\sigma_{ 4 } x_1+  \sigma_{ 3 } x_2  -\sigma_{ 2 }x_3+\sigma_{ 1 }x_4) \mathbf{k}
\end{split}\end{equation}
we define the associated matrix
\begin{equation}\label{eq:sigma-R}
 \widetilde{\sigma}^{\mathbb{R}} :=     \left(
\begin{array}{rrrr }\sigma_{ 1  }&-  \sigma_{  2 }&- \sigma_{ 3  }&- \sigma_{4   }  \\  \sigma_{2   }&   \sigma_{1   }&  \sigma_{ 4  }&- \sigma_{  3 } \\  \sigma_{ 3  }&-  \sigma_{  4 }&  \sigma_{ 1  }&  \sigma_{ 2  } \\ \sigma_{ 4  }&  \sigma_{ 3  }&- \sigma_{ 2  }&  \sigma_{  1 } \end{array}\right).
\end{equation}
Thus  (\ref{eq:q-sigma}) can be written as
\begin{equation}\label{eq:widetildeX-q}
      (q_l\sigma)^{\mathbb{R}}= \widetilde{\sigma}^{\mathbb{R}} q_l^{\mathbb{R}},
\end{equation}
 for $\widetilde{\sigma}^{\mathbb{R}}$   given by (\ref{eq:sigma-R}). It follows from (\ref{eq:sigma-R}) that
$    \widetilde{\overline{\sigma}}^{\mathbb{R}}=\left(\widetilde{\sigma}^{\mathbb{R}}\right)^t
$, where $ \overline{ \sigma}$  is the conjugate of~$ {\sigma}$.

Denote $(y_1,\ldots,y_4)^t=\widetilde{\sigma}^{\mathbb{R}}(x_1,\ldots,x_4)^t$, i.e. $y_k=\sum_{k=1}^4\widetilde{\sigma}^{\mathbb{R}}_{kj}x_j$, $k=1,\ldots,4$. Since $\partial_{x_j}[ f(\ldots,$ $\widetilde{\sigma}^{\mathbb{R}}q_j^{\mathbb{R}}, \ldots)]=\sum_{k=1}^4 \widetilde{\sigma}^{\mathbb{R}}_{kj}\partial_{y_k}f(\ldots,y,\ldots)$, we find that
 \begin{equation*}\begin{split}
    \overline{\partial}_{q_l}[ f( \ldots,\widetilde{\sigma}^{\mathbb{R}}q^{\mathbb{R}},\ldots )]&=\sum_{j =1}^4 \mathbf{i}_{j } \partial_{x_j}[f( \ldots,\widetilde{\sigma}^{\mathbb{R}}q^{\mathbb{R}},\ldots)]\\&=\sum_{j,k=1}^4 \mathbf{i}_{j }\widetilde{\sigma}^{\mathbb{R}}_{kj}\partial_{y_k}f( \ldots,y,\ldots )
    =\sum_{j,k  =1}^4 \mathbf{i}_{j } \widetilde{\overline{\sigma}}^{\mathbb{R}}_{jk}\partial_{y_k}f( \ldots,y,\ldots)\\&
    =  \overline{\partial}_{Q_l }( \overline{\sigma} f)( \ldots, Q_l,\ldots ),
\end{split} \end{equation*}
by
$
   \Big ( \sum_{j =1}^4 \mathbf{i}_j\partial_{y_j}\Big)\overline{ {\sigma}}=\sum_{j,k=1}^4 \mathbf{i}_j\widetilde{\overline{\sigma}}^{\mathbb{R}}_{jk}\partial_{y_k}
$
and~\eqref{eq:widetildeX-q}.
\end{proof}

\begin{cor}
If $f$ is regular, then $\hat f=f(\mathbf{a}  q )$  for some $ \mathbf{a}\in {\rm
GL}( n ,\mathbb H)$ and $\widetilde f=f(  q \sigma )$ for some $ \sigma\in  \mathbb H $ are both regular.
\end{cor}

\begin{cor}\label{cor:transformations}  The space of all regular functions on $\mathcal{U}_n$ is invariant under the transformations defined in Propositions~\ref{prop:transformations}.  Namely, if $f$ is regular on the   Siegel upper half space $\mathcal{U}_n$, then the functions
$f( \tau_p (q)) $, $p\in \partial\mathcal{U}_n$; $f( R_{\mathbf{a}}( q) )$,
$ \mathbf{a}\in {\rm Sp}(n-1)$;  $\sigma f( R_{ \sigma  }( q))$
 for some $\sigma\in \mathbb{H}$ with $|\sigma|=1$, and
$f( \delta_r( q) )$ are all regular on  $\mathcal{U}_n$.
\end{cor}
\begin{proof}

The translate $\tau_p$ in (\ref{eq:translates}) can be represented as a composition of the linear transformation given by the quaternionic matrix
\begin{equation*}
     \left[\begin{array}{cc}1&2\overline{p}'\\0&I_{n-1}\end{array}\right],
\end{equation*}
and the Euclidean translate $(q_1,q')\mapsto (q_1+p_1,q'+p')$. The first transformation preserves the regularity of a function by Proposition \ref{prop:D-A}, while the later one   obviously preserves the regularity of a function since
the  Cauchy-Fueter operators are of constant coefficients.

The equation
\begin{equation*}
     \overline{\partial}_{q_l}[\sigma f(q\sigma)]=   \overline{\partial}_{Q_l}[\overline{\sigma}\sigma f(Q)]_{Q=q\sigma}= |\sigma|^2 \overline{\partial}_{Q_l}  f(Q)|_{Q=q\sigma}=0,
\end{equation*}
follows from Proposition~\ref{prop:D-A} (2) and shows that $\sigma f(\overline{\sigma} q_1\sigma,q'\sigma)$ is regular.

The rest of the corollary is obvious.
\end{proof}


\section{\bf Hardy space $H^2(\mathcal U_n)$}\label{sec:Hardy}


This section is devoted to the     properties of Hardy space on $\mathcal U_n$.
The identification of the quaternionic Heisenberg group and the boundary of the quaternionic Siegel upper half space allows to define the Lebesgue measure $d\beta (\cdot)$ on $\partial \mathcal U_n$ by pulling back by the  projection $\pi$~\eqref{eq:proj} the Haar measure on $qH^{n-1}$. The latter measure, in its term, is a pull back of the Lebesgue measure $d\mu(\cdot)=dx\,dq'$ from $\mathbb R^3\times\mathbb R^{4(n-1)}$. Let $ L^2(\partial\mathcal{U}_n )$ denote the  space of all  $\mathbb{H}$-valued functions which are square integrable with respect to the measure $d\beta$. It is easy to see by definition that $ L^2(\partial\mathcal{U}_n )$ is a
right quaternionic Hilbert space with the following inner product:
\begin{equation} \langle f,g\rangle_{L^2}=\int_{\partial\mathcal{U}_n }\overline{f(q)}g(q)d\beta(q).
\end{equation}

A function $F\in H^2(\mathcal{U}_n )$ has boundary value $F^b$ that belongs to $L^2(\partial\mathcal U)$ in the following sense.

\begin{thm}\label{thm:boundary_value} Suppose that $F\in H^2(\mathcal U_n)$. Then
\begin{itemize}
\item[1.]{There exists a function $F^b\in L^2(\partial\mathcal U_n)$ such that $F(q+\varepsilon \mathbf e)\vert_{\partial\mathcal U_n}\to F^b(q)$ as $\varepsilon\to 0$ in $L^2(\partial\mathcal U_n)$ norm.}
\item[2.]{$\|F^b\|_{L^2(\partial\mathcal U_n)}=\|F\|_{H^2(\mathcal U_n)}$,}
\item[3.]{The space of all boundary values forms a closed subspace of the space $L^2(\partial\mathcal U_n)$.}
\end{itemize}
\end{thm}
\begin{proof}
This theorem was proved in~\cite[Theorem 4.2]{CM08} for $n=2$. The arguments work for an arbitrary $n$ if we consider the following slice functions. Let $\mathcal H^2(\mathbb R^4_+)$ be the classical Hardy space, that is the set of all harmonic functions $u\colon \mathbb R^4_+\to\mathbb R$ such that
$$
\sup_{t>0}\|u(t,\cdot)\|_{L^2(\mathbb R^3)}<\infty.
$$
Assume that $F=F_1+\bfi F_2+\bfj F_3+\bfk F_4\in H^2(\mathcal U_n)$. Then the slice function $f_j(q_1):=F_j(q_1+|q'|^2,q')$ is
harmonic by (\ref{eq:pluriharmonicity-hat}), and belongs to $\mathcal H^2(\mathbb R^4_+)$ for each $j=1,\ldots,4$ and any fixed $q'\in\bH^{n-1}$. We omit further details.
\end{proof}

\begin{prop} \label{prop:Hardy} The Hardy space $H^2(\mathcal{U}_n )$ is
a right quaternionic  Hilbert space  under the inner product $\langle F,G\rangle=\langle F^b,G^b\rangle_{L^2(\partial\mathcal{U}_n )}$.
\end{prop}
\begin{proof}
Since  the Cauchy-Fueter operator $\overline\partial_{{q}_{l  }}$ in (\ref{eq:Cauchy-Fueter}) is right quaternionic linear, i.e., for a fixed $\sigma$
 \begin{equation*}
     \overline\partial_{{q}_{l  }}(f(q)\sigma)=(\overline\partial_{{q}_{l  }} f(q))\sigma,
 \end{equation*}
we see that
   $f(q)\sigma$  is regular if $f(q) $ is. Thus, the Hardy space $H^2(\mathcal{U}_n )$ is
a right quaternionic vector space.

 Set
\begin{equation}\label{eq:Cauchy-Fueter--}
  \partial_{{q}_{l+1 }} f:= \overline{\overline\partial_{{q}_{l+1 }} \overline{f}}=\partial_{x_{4l+1}}f-
 \partial_{x_{4l+2}}f\mathbf{i}
 - \partial_{x_{4l+3}}f\mathbf{j}- \partial_{x_{4l+4}}f\mathbf{k}.
\end{equation}
It is straightforward to see that
\begin{equation}\label{eq:pluriharmonicity-hat}
    0= {\partial}_{q_{l+1 } }\overline{\partial}_{q_{l+1 }} {f}= (\partial_{x_{4l+1}}^2+\partial_{x_{4l+2}}^2+
 \partial_{x_{4l+3}}^2+
\partial_{x_{4l+4}}^2)  {f} .
\end{equation}
  Consequently, $ {f}_1,\ldots,  {f}_4$  are harmonic on the line $\{(0,\dots q_l,\dots,0)\}$ and so  $f_1,\ldots, f_4$  are harmonic on  $\mathbb H^n$. Thus for $q\in \mathcal{U}_n$,
  \begin{equation*}\label{eq:mean}
      f_j(q)=\frac 1{|B|}\int_B  f_j(p) dV(p),\qquad j=1,2,3,4,
  \end{equation*}
  where $B$ is a small ball centered at $q$ and contained in $\mathcal{U}_n$, from which we see that
   \begin{equation}\label{eq:mean-||}
     | f (q)| \leq \frac 1{|B|}\int_B  |f (p)| dV(p)\leq \left(\frac 1{|B|}\int_B  |f (p)|^2 dV(p)\right)^{\frac 12}.
  \end{equation}
There exist $a,b>0$ such that $B\subset\mathcal{U}_{n;a,b}:=\{q\in \mathcal{U}_n\vert\ a < \re q_1-|q'|^2< b\}$, and so
\begin{equation} \label{eq:estimate}
\begin{split}
     | f (q)|^2 &\leq \frac 1{|B|}\int_{\mathcal{U}_{n;a,b}}  |f (x_1 ,\cdots ,  x_{4n} )|^2 dx_1\cdots dx_{4n}\\&\leq \frac 1{|B|}\int_{(a,b)\times \mathbb{R}^{4n-1}}  \left|f\left(x_1+\sum_{j=5}^{4n}|x_j|^2,x_2\cdots ,  x_{4n}\right)\right|^2 dx_1dx_2\cdots dx_{4n} \\
    & \leq \frac 1{|B|}\int_a^bdx_1\int_{\partial\mathcal{U}_{n } } |f  (p+ x_1  \mathbf{e})|^2 d\beta(p)
     \leq c \|f\|_{H^2(\mathcal U_n)}^2,
\end{split}  \end{equation}
  where $c=(b-a)/{|B|}$ is a positive constant depending on $q$, and independent of the functions $f\in H^2(\mathcal{U}_n )$. Here  we have used the coordinates transformation $(x_1 ,\cdots ,  x_{4n} )\rightarrow(x_1+\sum_{j=5}^{4n}|x_j|^2,x_2\cdots ,  x_{4n}) $, whose  Jacobian is identity.

To prove the completeness, we suppose that a Cauchy sequence $\{f^{(k)}\}$ in the Hardy space $H^2(\mathcal{U}_n )$ is given. We need to show that some subsequence converges
to an element in $H^2(\mathcal{U}_n )$. Apply the estimate (\ref{eq:estimate}) to regular functions $f^{(k)}  -f^{(l)}$ to get that for any compact subset $K\subset\mathcal U_n$ and  $q\in K$,
\begin{equation} \begin{split}
     | f^{(k)} (q)-f^{(l)}(q)|
     \leq c_K \|f^{(k)}  -f^{(l)}\|_{H^2(\mathcal U_n)},
\end{split}  \end{equation}
 where $c_K$ is a positive constant only depending on $K$. It means that the sequence $\{f^{(k)}\}$ converges uniformly on any compact subset of  $\mathcal{U}_n $. Denote by  $f$   the limit. Recall the well known estimate
 \begin{equation}\label{eq:C1-estimate}
     \|u\|_{C^1(B(q,r))}\leq C_r\|u\|_{C^0(B(q,2r))}
 \end{equation}
for any harmonic function $u$ defined on the ball $B(q,2r)$,  where $C_r$ is a positive constant only depending on $r$ and the dimension, and independent of a function $u$ (see {\it e.g.,} pp. 307-312 in \cite{Stein-F}).
Now apply the estimate~\eqref{eq:C1-estimate} to each component of regular function $f=f_1+
 \mathbf{i}f_2
 + \mathbf{j}f_3+ \mathbf{k}f_4$, which is harmonic. By the argument of finite covering and estimate (\ref{eq:estimate}), we easily see that
  \begin{equation*}
     \|f\|_{C^1(K)}\leq C_K'\|f\|_{H^2(\mathcal{U}_n)}
 \end{equation*}
 for some constant $C_K'$ only depending on $K$. It follows that $ |\partial_{x_j} f^{(k)} (q)-\partial_{x_j} f^{(k)}(q)|
     \leq C_K' \|f^{(k)}  -f^{(l)}\|_{H^2(\mathcal U_n)}$ for $q\in K$, $j=1,\ldots 4n$. Consequently, the limit function $f$ is also $C^1$ and  $\lim_{k\rightarrow\infty}\partial_{x_j} f^{(k)}(q)=\partial_{x_j}f(q) $. Thus, $\overline\partial_{{q}_{l  }} f(q)
     =\lim_{k\rightarrow\infty}\overline\partial_{{q}_{l  }}  f^{(k)}(q)=0 $. Namely, the limit function $f$ is regular.

     Since on the compact subset $K_{R,\varepsilon}:=\partial\mathcal{U}_n \cap  \overline{ B(0,R)}+\varepsilon\mathbf{e}$ for fixed $R,\varepsilon>0$, the sequence $\{f^{(k)}\}$ is uniformly convergent, we find that
     \begin{equation*}
     \begin{split}
     \int_{\partial\mathcal{U}_n \cap   \overline{B(0,R)} }|f_\varepsilon(q)|^2d\beta(q)&=\int_{K_{R,\varepsilon}}|f(q)|^2d\beta(q)
     \\&=
          \lim_{k\rightarrow\infty}   \int_{K_{R,\varepsilon}}|f^{(k)}(q)|^2d\beta(q)\leq \sup_k\|f^{(k)}\|_{H^2(\mathcal{U}_n)}<\infty.
\end{split}    \end{equation*}
  Consequently, $f_\varepsilon$ is square integrable on $\partial\mathcal{U}_n$ for any $\varepsilon>0$, and  $ \int_{ \partial\mathcal{U}_n  }|f_\varepsilon(q)|^2d\beta(q) \leq \sup\|f^{(k)}\|_{H^2(\mathcal{U}_n)}$. Thus $f\in H^2(\mathcal{U}_n )$.
\end{proof}

\begin{prop}\label{pr:Hardy_invariance} The Hardy space
$H^2(\mathcal{U}_n )$ is invariant under the transformations of Proposition~\ref{prop:transformations}.
\end{prop}
\begin{proof}
Since the regularity property  and  the hypersurface   $\partial\mathcal U_n+\varepsilon\mathbf{e}$ for each $\varepsilon>0$ are invariant under these transformations by Corollary~\ref{cor:transformations} and the measure $d\beta$ either invariant or has a finite distortion, the proof follows.
\end{proof}


\section{\bf The Cauchy-Szeg\"o kernel}\label{sec:kernel}


In this section we introduce the notion of the Cauchy-Szeg\"o kernel for the projection operator from  the space $L^2(\partial \mathcal U)$ to the space of the boundary values of function from Hardy space $H^2(\mathcal U_n)$. We study its properties, particularly showing that it is invariant under translates, rotations and dilations defined in Proposition~\ref{prop:transformations} and, finally, we present the formula for the Cauchy-Szeg\"o kernel.


\subsection{Existence and characterization of the Cauchy-Szeg\"o kernel}


\begin{thm}\label{thm:CM08}  The Cauchy-Szeg\"o
kernel $S(q, p)$ is a unique $\mathbb H$-valued function, defined on $\mathcal{U}_n  \times \mathcal{U}_n$  satisfying the following
conditions.
\begin{itemize}
\item[1.]{ For each $p\in \mathcal{U}_n$, the function $q \mapsto S(q, p)$ is regular for $q\in \mathcal{U}_n$, and belongs
to $H^2(\mathcal{U}_n)$. This allows to define the boundary value $S^b(q, p)$ for each $p\in \mathcal{U}_n$ and for
almost all $q\in \partial\mathcal{U}_n$.}
\item[2.]{ The kernel $S$ is symmetric: $S(q, p) = \overline{S(p, q)}$ for each $(q, p)\in\mathcal{U}_n  \times \mathcal{U}_n$. The symmetry
permits to extend the definition of $S(q ,p )$ so that for each $q\in \mathcal{U}_n$, the function
$S_b(q, p)$ is defined for almost every $p\in \partial\mathcal{U}_n$ $($here we use the subscript  $b$ to indicate
the boundary value with respect to the second argument$)$.}
\item[3.]{ The kernel $S$ satisfies the reproducing property in the following sense
 \begin{equation}\label{eq:Szego-2}
    F(q) =
\int_{\partial\mathcal{U}_n}
S_b(q, Q)F^b(Q) d\beta (Q),  \qquad q\in \mathcal{U}_n,
 \end{equation}
whenever $F \in   H^2(\mathcal{U}_n)$.
}
\end{itemize}
\end{thm}

\begin{proof}
 We must
  show that the Hardy space $H^2(\mathcal U_n)$ is nontrivial    first. Otherwise the  Cauchy-Szeg\"o
kernel vanishes.  We claim that  the function $s\left(q_1+\overline{p}_1  -2\sum_{k=2}^{n} \bar p_k'q'_k \right)$ for fixed $(p_1,\ldots, p_n)\in \mathcal{U}_n$, with $s(\cdot)$ given by   (\ref{eq:s}),   is in the Hardy space $H^2(\mathcal U_n)$.

  Apply the Laplace operator
\begin{equation} \label{eq:Laplace} \overline{\partial}_{q_{1}}{\partial}_{q_{1}} = \partial_{x_{1}}^2+\partial_{x_{ 2 }}^2+
\partial_{x_{3}}^2+\partial_{x_{4}}^2   
\end{equation}
to the harmonic function $ \frac {1}{|q_1|^2}$ on $\mathbb{H}\setminus\{0\}$ to see that
  $ {\partial}_{q_{1}}\frac {1}{|q_1|^2}= -\frac {2\overline{q}_{1}}{|q_{1}|^4}=h(q_1)$ is a regular function on $\mathbb{H}\setminus\{0\}$,  which  
  is homogeneous of degree $-3$.  Since $\frac {\partial^{2n}}{\partial x_1^{2n}}$ commutes with $\overline{\partial}_{q_{1}}$, the function $s( q_1  )= c_n  \frac {\partial^{2n}}{\partial x_1^{2n}}\frac {\overline{q_1} }{|q_1|^4}$  in (\ref{eq:s}) is regular on $\mathbb{H}\setminus\{0\}$. 
  Consequently, $s( q_1 +x_1 )$ for fixed $x_1>0$ ia also regular on $\mathbb{H}\setminus\{-x_1\}$, and so $\widetilde{s}(q_1,\ldots,q_n):=s( q_1 +x_0 )$ is  
  regular on $\big(\mathbb{H}\setminus\{-x_1\}\big) \times \mathbb{H}^{n-1}$. 
  In particular, $\widetilde{s}( \cdot) $ is  regular on the quaternionic Siegel upper half space $ \mathcal{U}_n$. Now by the invariance in Corollary \ref{cor:transformations}, we see that $\widetilde{s}(\tau_{p^{-1}}(q))$ is also  regular  for fixed $p=(x_2\mathbf{i}+x_3\mathbf{j}+x_4\mathbf{k},p')\in\partial \mathcal{U}_n$. So
$s\left(q_1+\overline{p}_1  -2\sum_{k=2}^{n} \bar p_k'q'_k \right)=\widetilde{s}(\tau_{p^{-1}}(q))$ with $p_1=x_1 +x_2\mathbf{i}+x_3\mathbf{j}+x_4\mathbf{k}$ is regular.

Note that there exists a   constant $C>0$, only depending on the dimension $n$,  such that
\begin{equation*}
    | \widetilde{s}_\varepsilon ( q )|^2\leq \frac C{|q_1 +x_0 +\varepsilon|^{4n+6}}=\frac C{ ((|q'|^2+x_0)^2   + |{\rm Im}q_1|^2)^{2n+3}}
\end{equation*}
for  $q\in \partial \mathcal{U}_n$, which is obviously integrable with respect to the measure $d\beta$. Namely, $\widetilde{s} ( \cdot )$ is in Hardy space $H^2(\mathcal U_n)$, and so is $\widetilde{s}(\tau_{p^{-1}}(q))$ by
the invariance of the Hardy space under the translates in  Proposition \ref{pr:Hardy_invariance}. The claim is proved.

Now for fixed $q\in \mathcal{U}_n$, define a quaternion-valued  right linear functional
  \begin{equation}\label{eq:l-q}\begin{split}
   l_q:H^2(\mathcal{U}_n)&\longrightarrow \mathbb{H}, \\
   F&\longmapsto  F(q).
\end{split}\end{equation}
 It is  bounded  by estimate (\ref{eq:estimate}).
Apply the quaternion version of Riesz's representation theorem to see that there exists an element, denoted by $K (\cdot,q)\in H^2(\mathcal{U}_n)$ such that $l_q( F)= \langle K (\cdot,q), F\rangle=\langle K^b (\cdot,q), F^b\rangle_{L^2(\partial \mathcal U_n)}$. Here $K  (\cdot, \cdot)$ is nontrivial and the boundary value $K^b(p, q)$ exists for
almost all $p\in \partial\mathcal{U}_n$.    We have
  \begin{equation}\label{eq:kernel}
  F(q)= \int_{\partial\mathcal{U}_n}
\overline{K^b(  Q, q)} F^b(Q) d\beta (Q).
  \end{equation}

For fixed $p\in \mathcal{U}_n$, applying (\ref{eq:kernel}) to $K (\cdot,p)$ and $K (\cdot,q)$, we see that
\begin{equation*}\begin{split}
     K (q,p)&=( K^b (\cdot,q), K^b (\cdot,p) )=\int_{\partial\mathcal{U}_n} \overline{ K^b (Q,q)} K^b (Q,p)  d\beta (Q)\\ & =\overline{\int_{\partial\mathcal{U}_n} \overline{K^b (Q,p)} K^b (Q,q)   d\beta (Q)} =\overline{  K ( p,q)   }.
\end{split}\end{equation*}

Denote $S(q, p):=\overline{K (p,q)}$ for $(q, p)\in\mathcal{U}_n  \times \mathcal{U}_n$. Then  $S(q, p)=K (q, p )  $ is regular in $q$, and $S(q, p) =\overline{K (p,q)}= \overline{S(p, q)}$. The function $S$  has the boundary values  as in Theorem~\ref{thm:boundary_value}. Moreover, we have
\begin{equation}\label{eq:boundary-symmetry}
    S_b (Q,p )=\overline{S^b (p,Q)}
\end{equation}
  for $p\in \mathcal{U}_n$, $Q\in \partial\mathcal{U}_n$, which follows from the symmetry  $S(q+\varepsilon\mathbf{e}, p) = \overline{S(p, q+\varepsilon\mathbf{e})}$
  by taking   $\varepsilon\rightarrow0+$.

To show the uniqueness, suppose that $\widetilde{S}( \cdot , \cdot)$ is another function satisfying Theorem~\ref{thm:CM08}. By definition $\widetilde{S}(\cdot  , q)\in H^2(\mathcal{U}_n)$ for any fixed $q\in \mathcal{U}_n$. Choose an arbitrary $p\in \mathcal{U}_n$ and apply the reproducing formula~\eqref{eq:Szego-2} of $S(\cdot  , \cdot)$ and $\widetilde{S}(\cdot  , \cdot)$ to get
\begin{equation*}
\begin{split}
     \widetilde{S} (p,q)&=  \int_{\partial\mathcal{U}_n} S_b (p,Q) \widetilde{S}^b(  Q, q)   d\beta (Q)
     =\overline{\int_{\partial\mathcal{U}_n}\overline{\widetilde{S}^b(  Q, q)}    \overline{S_b (p,Q)}  d\beta (Q)}
     \\ & =\overline{\int_{\partial\mathcal{U}_n} \widetilde{S}_b( q, Q  ) S^b (Q, p )  d\beta (Q)}
     =\overline{   S (q, p ) }= S (p,q   ).
\end{split}
\end{equation*}
In the third identity, we used the equation~\eqref{eq:boundary-symmetry} for $S(\cdot  , \cdot)$ and $\widetilde{S}(\cdot  , \cdot)$.
The theorem is proved.
 \end{proof}

The function $ S(q, p)$ is conjugate right regular in variables $p=(p_1,\ldots,p_n)$:
\begin{equation*}
      {\partial}_{p_l} S(q, p)=\overline{\overline{\partial}_{p_l}K (p,q)}=0.
\end{equation*}


\subsection{\bf Invariance of the Cauchy-Szeg\"o kernel}


Since the Siegel upper half space possesses some invariance properties, it is expected that the Cauchy-Szeg\"o kernel also inherits them. Namely, the following proposition is true.

\begin{prop}\label{prop:Inv-kernel} The Cauchy - Szeg\"o kernel has following invariance properties. \begin{equation}\label{eq:Inv-kernel}
\begin{split}&   S(\tau_p (q),  \tau_p (Q))= S(q,Q),\\&
  S(  R_{\mathbf{a}} ( q),  R_{\mathbf{a}} (   Q ))= S(q,Q),\\ &\sigma S( R_{ \sigma  }( q)  ,   R_{ \sigma  }( Q ))\overline{\sigma}= S(q,Q),\\&
 S(  \delta_r( q) ,\delta_r(  Q ))r^{4n+6}= S(q,Q).
   \end{split}
   \qquad\text{for}\quad q,Q\in \mathcal{U}_n,
\end{equation}
where $p\in \partial\mathcal{U}_n$,
 $\mathbf{a}\in {\rm Sp}(n-1)$,   $\sigma\in \mathbb{H}$ with $|\sigma|=1$  and $r>0$.
\end{prop}
\begin{proof}  Note that the measure $ d\beta (Q)$ is invariant under the translate $\tau_p$. If $F \in   H^2(\mathcal{U}_n)$, then $ F(\tau_{-p} (q))\in   H^2(\mathcal{U}_n)$ by Proposition~\ref{pr:Hardy_invariance}. We get
\begin{equation*}
    F( \tau_{-p} (q) ) =
\int_{\partial\mathcal{U}_n}
S_b(q,  Q)F^b(\tau_{-p} (Q)) d\beta (Q)=
\int_{\partial\mathcal{U}_n}
S_b(q,  \tau_p (Q))F^b( Q ) d\beta (Q),
\end{equation*}
  and  by substituting $\tau_{-p} (q)\mapsto q$, we obtain
 \begin{equation*}
    F(   q ) =
\int_{\partial\mathcal{U}_n}
S_b(\tau_p (q),\tau_p (Q)  )F^b( Q) d\beta (Q).
 \end{equation*}
 We conclude that the function $S(\tau_p (q),\tau_p (Q)  )$ is also regular in $q$ by Corollary~\ref{cor:transformations} and it is symmetric.
The first identity in~\eqref{eq:Inv-kernel} follows by the uniqueness in Theorem~\ref{thm:CM08}.

It follows from $|\xi\sigma|=|\sigma\xi|=|\xi|$ for any quaternionic number $\xi\in \mathbb{H} $ that $(q_1 , q'     )\mapsto (q_1 \sigma, q'\sigma   )$ and $(  q_1  ,q'   )\mapsto (\overline{\sigma} q_1  ,q'   )$ are both orthogonal maps, so is their composition  $ R_{\sigma}\colon ( q_1 ,q')\mapsto (\overline{\sigma}q_1 \sigma , q'\sigma   )$. If $F \in   H^2(\mathcal{U}_n)$, then $  \sigma^{-1} F( R_{\sigma^{-1}}(q)  )$ is regular by Corollary~\ref{cor:transformations} and is in $    H^2(\mathcal{U}_n)$ by definition. Therefore,
 \begin{equation*}\begin{split}
   \sigma^{-1} F( R_{\sigma^{-1}}(q)  )& =
\int_{\partial\mathcal{U}_n}
S_b(q,  Q )\sigma^{-1} F^b( R_{\sigma^{-1}}(Q)  ) d\beta (Q)\\&
 =
\int_{\partial\mathcal{U}_n}
S_b(q,  R_{\sigma }(Q)   )\sigma^{-1} F^b( Q  ) d\beta (Q),
\end{split} \end{equation*}
since  $ d\beta$ is invariant under the orthogonal transformation  $R_{\sigma }$. Substituting $R_{\sigma^{-1}}(q)$ $ \mapsto q$ and multiplying by $\sigma$ on both sides, we get
  \begin{equation*}
  F(  q     ) =
\int_{\partial\mathcal{U}_n}
\sigma S_b(R_{\sigma }(q),R_{\sigma }(Q))\overline{\sigma}  F^b(  Q   ) d\beta (Q).
 \end{equation*}
The function $\sigma S(R_{\sigma }(q),R_{\sigma }(Q))\overline{\sigma}$ is also regular in $q$ by Corollary~\ref{cor:transformations} again and it is symmetric.  The third identity in~\eqref{eq:Inv-kernel} follows by the uniqueness in Theorem~\ref{thm:CM08}.

The second and the fourth identities  are proved by the similar arguments.
\end{proof}


\subsection{\bf Determination of the Cauchy-Szeg\"o kernel}


It is sufficient to show $S_b(q,0)=s(q_1)$.  This is because
\begin{equation}\label{eq:S-s}
     S_b(q,p)=S_b( p^{-1}\cdot q,0)=s\left(q_1-\overline{p}_1  -2\sum_{k=2}^{n} \bar p_k'q'_k\right)
\end{equation}
for $p=( {p}_1 ,\ldots, p ')\in\partial\mathcal{U}_n$, $q\in \mathcal{U}_n$. Taking conjugate in both sides of (\ref{eq:S-s}), we see that \begin{equation}\label{eq:S-s'}
     S^b(q,p)  =s\left(q_1-\overline{p}_1  -2\sum_{k=2}^{n} \bar p_k'q'_k\right)
\end{equation}  holds for $p \in\mathcal{U}_n$ and $q\in \partial\mathcal{U}_n$  by the symmetry of the Cauchy-Szeg\"o
kernel $S(q, p)$ in Theorem \ref{thm:CM08}. Now we fix a point  $(p_1,\ldots, p_n)\in \mathcal{U}_n$. In the proof of Theorem \ref{thm:CM08}, 
we have seen that $s(q_1-\overline{p}_1  -2\sum_{k=2}^{n} \bar p_k'q'_k)$   is in the Hardy space $H^2(\mathcal U_n)$. As elements of the Hardy space $H^2(\mathcal U_n)$,   $S (\cdot,p)$ and $ s(q_1-\overline{p}_1  -2\sum_{k=2}^{n} \bar p_k'q'_k)$ coincide on the boundary $\partial\mathcal{U}_n$. 
They must coincide on the whole $\mathcal{U}_n$ by the uniqueness of the  Cauchy-Szeg\"o kernel following from the reproducing property (\ref{eq:Szego-2}).

Since by
\begin{equation}0=\sum_{l=2}^{n }{\partial}_{q_{l}}\overline{\partial}_{q_{l}}u(q_{1},q'
)=\sum_{l=2}^{n }(\partial_{x_{4l-3}}^2+\partial_{x_{4l-2 }}^2+
\partial_{x_{4l-1}}^2+\partial_{x_{4l}}^2)u(q_1, q'
),\label{eq:harmonicity}
\end{equation}
where $u(q)=S_b(q,0)$, each component of $u( q_1, \cdot)$ is a harmonic function on the ball $\{q'\in \mathbb{H}^{n-1}\mid\
|q'|<\re q_{1  }\}$ for fixed $q_1$ with $\re q_{1  }>0$. On the other hand,
\begin{equation}
  S_b(  (q_1,  \mathbf{a}q' ),   0)= S_b(( q_1,q' ),0)\quad\text{for}\quad q \in \mathcal{U}_n,
\end{equation}
by Proposition~\ref{prop:Inv-kernel}. Since ${\rm Sp}(n-1)$ acts on the sphere $\{q'\in \mathbb{H}^{n-1}\mid\
|q'|=R\}$ transitively, where $R<\re q_{1  }$,
we see that $S_b(( q_1,q' ),0)$ is constant on the sphere. Applying the maximum principle to each component of $S_b(( q_1,q' ),0)$ as a harmonic function in $q'$, we conclude that $S_b(( q_1,q' ),0)$ is constant on the ball $\{q'\in \mathbb{H}^{n-1}\mid
|q'|<\re q_{1  }\}$, and so $S_b((q_1, q' ),0)\equiv S_b(( q_1, 0 ),0)$. Denote $ s(q_1):=S_b(( q_1, 0 ),0)$, an $\mathbb{H}$-valued function defined on the half space $\mathbb{R}^4_+=\{q_1\in \mathbb{H}\mid\ \re q_{1  }>0\}$.

By the third identity in~\eqref {eq:Inv-kernel} we have
$\sigma S_b((  \overline{\sigma} q_1\sigma, 0 ),   0)\overline{\sigma}= S_b((    q_1 , 0),0)
$.
More precisely,
\begin{equation}\label{eq:s-1}
 s(  \overline{\sigma} q_1\sigma )= \overline{\sigma} s(  q_1) \sigma,
\end{equation}
 for any $\sigma\in \mathbb{H} $ with $|\sigma|=1$, and similarly
\begin{equation}\label{eq:s-2}
   s(r  q_1  )  = r^{-2n-3}  s(  q_1).
\end{equation}
by the fourth  identity in~\eqref {eq:Inv-kernel} and $\delta_r:(q_1,0)\mapsto (r^2q_1,0)$.

Take $q_1=x_1\in \mathbb{R}$ in  (\ref{eq:s-1}) to get $s( x_1  )= \overline{\sigma} s(  x_1) \sigma$. Write $s(  x_1)=\xi_1+\xi_2\mathbf{i}+\xi_3\mathbf{j}+\xi_4\mathbf{k}$ and choose $\sigma=\mathbf{i}$. Then $\xi_1+\xi_2\mathbf{i}+\xi_3\mathbf{j}+\xi_4\mathbf{k} =\overline{\mathbf{i}}(\xi_1+\xi_2\mathbf{i}
+\xi_3\mathbf{j}+\xi_4\mathbf{k}){\mathbf{i}}=\xi_1+\xi_2\mathbf{i}-\xi_3\mathbf{j}-\xi_4\mathbf{k}$, and so $\xi_3=\xi_4=0$. Similarly, $\xi_2=0$ by choosing $\sigma=\mathbf{j}$. Thus,~\eqref{eq:s-1} implies that $s(  x_1)$ must be real.

Note that
\begin{equation}\label{eq:orbit}
     \overline{\sigma}(x_1+\mathbf{i}x_2)\sigma=x_1+ x_2[(2y_2^2-1)\mathbf{i}+2y_2y_3\mathbf{j}+2y_2y_4\mathbf{k}],
\end{equation}
if $\sigma=y_2\mathbf{i}+y_3\mathbf{j}+y_4\mathbf{k}$ with $|\sigma|=1$. It easily follows from  (\ref{eq:orbit})  that the orbit of  $x_1+\mathbf{i}x_2$ under the adjoint action of unit quaternions is the 2-dimensional sphere
\begin{equation*}
     \{x_1+\xi_2\mathbf{i}+\xi_3\mathbf{j}+\xi_4\mathbf{k}; \xi_2^2+\xi_3^2+\xi_4^2= x_2^2\}.
\end{equation*}
Hence $s(q_1)$ is determined by its values on $\mathbb{R}^2_+=\{(x_1,x_2)\in \mathbb{R}^2;x_1>0\}$ by (\ref{eq:s-1}).
The homogeneous degree of $s$ in (\ref{eq:s-2}) implies that $s(q_1)$ is determined by its values in the semicircle $ \{(x_1,x_2)\in \mathbb{R}^2;x_1>0,x_1^2+x_2^2=1 \}$. At last, the  Cauchy-Fueter equation for $s$ gives four ordinary differential equations for four components of $s $ along the semicircle.
These ODEs together with the value $s(1)$ uniquely determine the function $s$.

\begin{prop}\label{prop:unique}  On the half space $\mathbb{R}^4_+=\{q_1\in \mathbb{H} \mid\ \re q_{1  }>0\}$, there exists a unique regular function  up to a real constant
satisfying~\eqref{eq:s-1}-\eqref{eq:s-2}.
\end{prop}
\begin{proof}
Since the conjugation action of unit quaternions leaves the function $s$ invariant, see~\eqref{eq:s-1}, its infinitesimal action coincides. From one side, choose $\sigma_t=\cos t +\sin t\mathbf{j}$ for small $t $. Then
\begin{equation*}\begin{split}
\overline{\sigma}_t q_1\sigma_t
&= q_1-t\mathbf{j} (x_1+x_2 \mathbf{i}+x_3 \mathbf{j}+x_4\mathbf{k})+t(x_1+x_2 \mathbf{i}+x_3 \mathbf{j}+x_4\mathbf{k})\mathbf{j}+O(t^2)\\&= q_1+2t  (-x_4 \mathbf{i} + x_2\mathbf{k} ) +O(t^2),
\end{split}\end{equation*} where $q_1=x_1+x_2 \mathbf{i}+x_3 \mathbf{j}+x_4\mathbf{k}$, from which we get
\begin{equation*}
 \left.\frac d{dt}\right|_{t=0}  s(  \overline{\sigma}_t q_1\sigma_t ) =-2x_4\partial_{x_2}s(  q_1 )+2x_2\partial_{x_4}s(  q_1).
\end{equation*}
From the other side, taking derivatives of  $\overline{\sigma}_t s(  q_1) \sigma_t$ with respect to $t$ at $0$ we get
\begin{equation}\label{eq:s-derivative-1}
  -2x_4\partial_{x_2}s(  q_1)+2x_2\partial_{x_4}s(  q_1) =-\mathbf{j}s(  q_1) + s(  q_1)\mathbf{j}.
\end{equation}
Similarly, choosing $\sigma_t=\cos t +\sin t\mathbf{k}$, we find that
\begin{equation}\label{eq:s-derivative-2}
    2x_3\partial_{x_2}s(  q_1)-2x_2\partial_{x_3}s(  q_1) =-\mathbf{k}s(  q_1)+s(  q_1)\mathbf{k}.
\end{equation}

The homogeneity of degree $-2n-3$ of the function $s$ in~\eqref{eq:s-2} implies the Euler equation for $s$:
\begin{equation}\label{eq:s-derivative-3}
     x_1\partial_{x_1}s(  q_1)+x_2\partial_{x_2}s(  q_1) +  x_3\partial_{x_3}s(  q_1)+x_4\partial_{x_4}s(  q_1)
     =-(2n+3)s(  q_1).
\end{equation}
Restricting to  $q_1=x_1+x_2\mathbf{i}\in \mathbb{R}^2_+$, i.e. $x_3=x_4=0$, we obtain
\begin{equation*}
      2x_2\partial_{x_4}s(  q_1) =-\mathbf{j}s(  q_1) + s(  q_1)\mathbf{j},\qquad-2x_2\partial_{x_3}s(  q_1)  =-\mathbf{k}s(  q_1)+s(  q_1)\mathbf{k}.
\end{equation*}
Substitute it  into the Cauchy-Fueter equation
\begin{equation*}
   \partial_{x_{ 1}}s(  q_1)+
 \mathbf{i}\partial_{x_{ 2}}s(  q_1)
 + \mathbf{j}\partial_{x_{ 3}}s(  q_1)+ \mathbf{k}\partial_{x_{ 4}} s(  q_1)=0
\end{equation*}
to deduce
\begin{equation}\label{eq:s-der}
   2x_2\partial_{x_{ 1}}s(  q_1)+
2x_2 \mathbf{i}\partial_{x_{ 2}}s(  q_1)
 =-2\mathbf{i}s(  q_1)+\mathbf{j}s(  q_1)\mathbf{k}-\mathbf{k}s(  q_1)\mathbf{j}.
\end{equation}
Write $s(x_{ 1}+
 \mathbf{i}x_{ 2})=f_1+
 f_2\mathbf{i}
 + f_3 \mathbf{j}+ f_4\mathbf{k}$ on $\mathbb{R}^2_+$. Then, the equation~\eqref{eq:s-der} is equivalent to
 \begin{equation}\label{eq:f-der}\begin{split}
     x_2(\partial_{x_{ 1}} f_1-
   \partial_{x_{ 2}} f_2)&=2f_2,
   \\
 x_2(\partial_{x_{ 1}} f_2+
  \partial_{x_{ 2}} f_1)&=0,
  \\
 x_2(\partial_{x_{ 1}} f_3-
   \partial_{x_{ 2}} f_4)&=  f_4,
   \\
  x_2(\partial_{x_{ 1}} f_4+
   \partial_{x_{ 2}} f_3)&=-  f_3.
\end{split} \end{equation}
Euler's equation~\eqref{eq:s-derivative-3} implies
\begin{equation}\label{eq:s-derivative-4}
     x_1\partial_{x_1}f_k+x_2\partial_{x_2}f_k
     =-(2n+3)f_k, \qquad k=1,2,3,4,
\end{equation} on  $\mathbb{R}^2_+$. Now  we have four real functions $f_1,f_2,f_3,f_4$  on the upper half plane $\mathbb{R}^2_+$ satisfying
 8 equations in~\eqref{eq:f-der}-\eqref{eq:s-derivative-4} with conditions $f_2(x_1,0)=f_3(x_1,0)=f_4(x_1,0)=0$ and $f_1(x_1,0)$ is real.

On the semicircle $ \{(x_1,x_2)\in \mathbb{R}^2\mid \ x_1>0,x_1^2+x_2^2=1 \}$ we can separate the system~\eqref{eq:f-der} into two parts, one of which depends on $(x_1,x_2)$ and another one depends on $(x_3,x_4)$ variables.
Take the sum of  the first identity in~\eqref{eq:f-der}, multiplying  by $x_2$, and the second one multiplying by $-x_1$ to get
\begin{equation}\label{eq:s-derivative-5}
    x_2( x_2\partial_{x_{ 1}}  -x_1
   \partial_{x_{ 2}}) f_1= x_2(2f_2+x_1\partial_{x_1}f_2 +x_2\partial_{x_2}f_2)=- (2 n+1)x_2f_2.
\end{equation}
Now   set $ x_1=\cos\theta$, $ x_2=\sin\theta$, $\theta\in (-\pi,\pi)$, and  $g_j(\theta):=f_j(\cos \theta, \sin\theta,0,0)$. The equality~\eqref{eq:s-derivative-5} implies
\begin{equation}\label{eq:g-theta1}
       g_1'(\theta)= (2n+1)g_2.
\end{equation}
Similarly, we have
\begin{equation*}\begin{split}&x_2( x_2\partial_{x_{ 1}}  -x_1
   \partial_{x_{ 2}}) f_2=2x_1f_2+ (2 n+3)x_2f_1,\\&x_2( x_2\partial_{x_{ 1}}  -x_1
   \partial_{x_{ 2}}) f_3= x_1f_3- (2 n+2)x_2f_4,\\&x_2( x_2\partial_{x_{ 1}}  -x_1
   \partial_{x_{ 2}}) f_4=  x_1f_4+ (2 n+2)x_2f_3,
\end{split}\end{equation*}
and so
\begin{equation}\label{eq:g-theta2}\begin{split}
  \sin\theta   g_2'(\theta)&=-2 g_2\cos\theta-(2n+3)g_1\sin\theta ,\\
\sin \theta   g_3'(\theta)&=-g_3\cos \theta   +2(n+1)g_4\sin\theta,\\  \sin \theta g_4'(\theta)&= - g_4\cos \theta-2(n+1)  g_3\sin\theta.
\end{split} \end{equation}
We obtain four real functions $g_1,g_2,g_3,g_4$  on $ (-\pi,\pi) $ satisfying
4 ordinary differential   equations~\eqref{eq:g-theta1}-\eqref{eq:g-theta2} under  the condition
\begin{equation}\label{eq:initial}
g_1(0)\in\mathbb{R}^1,\qquad    g_2(0)=g_3(0)=g_4(0)=0.
\end{equation}

To see that $g_3$ and $g_4$ vanishing, note that $s$ is real analytic since it is harmonic. So the functions $g_j$, $j=3,4$, are real analytic in $\theta$. Inductively, we can assume $g_j(\theta)=\sum_{m=N}^\infty a_m^{(j)}\theta^m$, $j=3,4$. Compare the coefficients of term $\theta^N$ in the third and fourth equations in~\eqref{eq:g-theta2}, we see that
\begin{equation*}
     Na_N^{(3)}=-a_N^{(3)},\qquad Na_N^{(4)}=-a_N^{(4)},
\end{equation*}
and so $a_N^{(3)}= a_N^{(4)}=0$. Therefore, $g_3\equiv g_4\equiv 0 $. The uniqueness of $g_1$ and $g_2$ follows from  vanishing of the solutions $g_1$, $g_2$  to~\eqref{eq:g-theta1} and the first equation in~\eqref{eq:g-theta2} with the initial conditions $g_1(0)=g_2(0)  =0$ by the same arguments as above.
  The result follows.
\end{proof}

\begin{cor}\label{cor:s} The function $s(q_{1})$ is given by
 \begin{equation}\label{eq:s-solution}
   c_n  \frac {\partial^{2n}}{\partial x_1^{2n}}\frac {\overline{q}_{1}}{|q_{1}|^4}
\end{equation}
for some real constant $ c_n$.
\end{cor}
 \begin{proof} In the proof of Theorem \ref{thm:CM08}, we have seen that
  $h(q_1)= {\partial}_{q_{1}}\frac {1}{|q_1|^2}= -\frac {2\overline{q}_{1}}{|q_{1}|^4} $ is a regular function on $\mathbb{H}\setminus\{0\}$, which   obviously  satisfies the invariance~\eqref{eq:s-1}, and so is  the function~\eqref{eq:s-solution}. The conjugation action $\overline{\sigma} q_1\sigma $, fixing $x_1$
 for any $\sigma\in \mathbb{H} $ with $|\sigma|=1$, implies
 \begin{equation*}\left(\frac {\partial^{2n}}{\partial x_1^{2n}}   h\right)(\overline{\sigma} q_1\sigma)=
  \frac {\partial^{2n}}{\partial x_1^{2n}}[   h(\overline{\sigma} q_1\sigma)]=
  \frac {\partial^{2n}}{\partial x_1^{2n}}[  \overline{\sigma} h( q_1)\sigma]=\overline{\sigma}\frac {\partial^{2n}}{\partial x_1^{2n}}  h( q_1 )\sigma,
 \end{equation*}
{\it  i.e.,}
  (\ref{eq:s-solution})
 satisfies the invariance~\eqref{eq:s-1}. The function, defined by~\eqref{eq:s-solution}, is homogeneous of degree $-2n-3$. So $s$ is given by
~\eqref{eq:s-solution}   by the uniqueness in Proposition~\ref{prop:unique}.

We verify now that $s(x_{ 1}+
 \mathbf{i}x_{ 2})$ satisfies~\eqref{eq:f-der}.
Write $s(x_{ 1}+
 \mathbf{i}x_{ 2})=f_1+
 f_2\mathbf{i}
 + f_3 \mathbf{j}+ f_4\mathbf{k}$ on $\mathbb{R}^2_+$. Then,  $f_3 \equiv0, f_4\equiv0$  and
 \begin{equation}\label{eq:f1-f2}
    f_1=  \frac {\partial^{2n}}{\partial x_1^{2n}}\frac {x_1}{(x_1^2+x_2^2 )^2},\qquad
      f_2=  \frac {\partial^{2n}}{\partial x_1^{2n}}\frac {-x_2}{(x_1^2+x_2^2 )^2},
 \end{equation}
 up to a constant $c_n$. Functions $f_j$'s satisfy (\ref{eq:f-der})-(\ref{eq:s-derivative-4}). Note that
 \begin{equation}
 \label{eq:partial-1}
     \partial_{x_{ 1}} \frac {-x_2}{(x_1^2+x_2^2 )^2}+ \partial_{x_{ 2}} \frac {x_1}{(x_1^2+x_2^2 )^2}
    = 0.
 \end{equation}
Taking  derivatives $ \frac {\partial^{2n}}{\partial x_1^{2n}}$ and multiplying by $x_2$ both sides of~\eqref{eq:partial-1}, one obtains
the second equation in~\eqref{eq:f-der}. Note that
 \begin{equation}
 \label{eq:partial-2}
     \partial_{x_{ 1}} \frac {x_1}{(x_1^2+x_2^2 )^2}-
   \partial_{x_{ 2}}  \frac {-x_2}{(x_1^2+x_2^2 )^2}=\frac { -2}{(x_1^2+x_2^2 )^2}.
 \end{equation}
Taking  derivatives $ \frac {\partial^{2n}}{\partial x_1^{2n}}$ and multiplying by $x_2$  on both sides of~\eqref{eq:partial-2},
we get the first equation in~\eqref{eq:f-der}.
\end{proof}


\subsection{Calculation of the constant for the Cauchy-Szeg\"o operator}


\begin{thm}
The constant $c_n$ in the function
$$s(q_1)=c_n\frac{\partial^{2n}}{\partial x_1^{2n}}\frac{\bar q_1}{|q_1|^4},\quad q_1=x_1+x_2i+x_3j+x_4k
$$
is given by~\eqref{eq:constant}.
\end{thm}

\begin{proof}
To calculate the constant $c_n$, we choose
$$
F(q)=F\big((q_1,q')\big)=c^{-1}_n\overline{S({\mathbf e},q)},\qquad {\mathbf e}=(1,0,\ldots,0).
$$
Then
$$
F(q)=\frac{\partial^{2n}}{\partial y_1^{2n}}\frac{\overline{1+\bar q_1-2\langle0,q'\rangle}}{|1+\bar q_1-2\langle0,q'\rangle|^4}=\frac{\partial^{2n}}{\partial y_1^{2n}}\frac{1+ q_1}{|1+\bar q_1|^4}\quad\text{ with}\quad y_1=1+x_1.
$$
First we calculate the value $F({\mathbf e})$. We obtain
$$
F({\mathbf e})=\frac{\partial^{2n}}{\partial x_1^{2n}}\frac{1+ q_1}{|1+\bar q_1|^4}\Big\vert_{x_1=1,x_{2,3,4}=0}=
\frac{\partial^{2n}}{\partial x_1^{2n}}\frac{1+ x_1}{|1+x_1|^4}\Big\vert_{x_1=1}
=\frac{\partial^{2n}}{\partial x_1^{2n}}\frac{1}{(1+x_1)^3}\Big\vert_{x_1=1}
$$
since $x_1>0$ in $\mathcal U_n$.
We continue
\begin{align*}
F({\mathbf e})& =\frac{\partial^{2n}}{\partial x_1^{2n}}\frac{1}{(1+x_1)^3}\Big\vert_{x_1=1}=
(-3)(-4)\ldots(-3-(2n-1))(1+x_1)^{-3-2n}\frac{(-2)}{(-2)}\Big\vert_{x_1=1}
\\
&=\frac{(-1)^{2(n+1)}}{2}(2n+2)!2^{{-2n-3}}=\frac{(2n+2)!}{2^{{2n+4}}}.
\end{align*}

On the other hand,
\begin{align*}
F({\mathbf e})& =\int_{\partial\mathcal U_n}S(e,Q)F^b(Q)d\beta(Q)
=c_n^{-1}\int_{\partial\mathcal U_n}S(e,Q)\overline{ S^b(e,Q)}d\beta(Q)
\\
&=
c_n\int_{\mathbb H^n}\int_{\mathbb R^3}\Big|\frac{\partial^{2n}}{\partial x_1^{2n}}\frac{1+ \bar q_1}{|1+q_1|^4}\Big|^2\,dq'\,dx_2dx_3dx_4,
\end{align*}
where $ q_1=|q'|^2+x_2\mathbf{i}+x_3\mathbf{j}+x_4\mathbf{k}$, since $Q$ is in $\partial\mathcal U_n$.

We start from calculating the derivative $\frac{\partial^{2n}}{\partial x_1^{2n}}\frac{\bar p}{|p|^4}$. We have
$$
\frac{\bar p}{|p|^4}=\frac{\bar p^{-1}\bar p}{|p|^2}p^{-1}=\bar p^{-1}p^{-2}\quad\text{by}\quad \frac{\bar p}{|p|^2}=p^{-1}.
$$
Thus
\begin{equation}\label{1}
\frac{\partial^{2n}}{\partial x_1^{2n}}\bar p^{-1}p^{-2}
=\sum_{k=0}^{2n}C^{k}_{2n}\frac{\partial^{k}}{\partial x_1^{k}}\bar p^{-1}\frac{\partial^{2n-k}}{\partial x_1^{2n-k}}p^{-2}.
\end{equation}
Since
$$
\frac{\partial^{k}}{\partial x_1^{k}}\bar p^{-1}=(-1)(-2)\ldots(-1-(k-1))\bar p^{-1-k}=(-1)^k
k!\bar p^{-1-k}
$$
and
$$
\frac{\partial^{2n-k}}{\partial x_1^{2n-k}}p^{-2}=(-2)(-3)\ldots(-2-(2n-k-1))p^{-2-2n+k}=
(-1)^{2n-k}(2n-k+1)!p^{-2-2n+k},
$$
substituting them in~\eqref{1}, we get
\begin{align*}
\frac{\partial^{2n}}{\partial x_1^{2n}}\bar p^{-1}p^{-2}& =\sum_{k=0}^{2n}\frac{(2n)!}{k!(2n-k)!}(-1)^{k}k!\bar p^{-1-k}(-1)^{2n-k}(2n-k+1)!p^{-2-2n+k}
\\
&=
(2n)!\sum_{k=0}^{2n}(2n-k+1)\bar p^{-1-k}p^{-2-2n+k}
\\
&=(2n)!\sum_{k=0}^{2n}(2n-k+1)\frac{\bar p^{-k}p^k}{|p|^2}p^{-2n-1}.
\end{align*}
Taking the absolute value, we get
$$
\Big|\frac{\partial^{2n}}{\partial x_1^{2n}}\bar p^{-1}p^{-2}\Big|^2=\frac{\big((2n)!\big)^2}{|p|^{4n+6}}\Big|\sum_{k=0}^{2n}(2n-k+1)\Big(\frac{p^2}{|p|^2}\Big)^k\Big|^2,\quad p=1+q_1.
$$

Now we concentrate in calculating the absolute value of the latter sum. Observe that it is square of the length of the sum of quaternions obtained by rotation on the same angle.
We denote by $\frac{(1+q_1)^2}{|1+q_1|^2}=s=\re s+\bfi s_2+\bfj s_3+\bfk s_4$, where
$$\re s=\frac{(1+x_1)^2-(x_2^2+x_3^2+x_4^2)}{(1+x_1)^2+x_2^2+x_3^2+x_4^2},\qquad \bfi s_2+\bfj s_3+\bfk s_4=\frac{2(1+x_1)(x_2\mathbf{i}+x_3\mathbf{j}+x_4\mathbf{k})}{(1+x_1)^2+x_2^2+x_3^2+x_4^2}.
$$
Note that for any quaternion $s$, written as $s=\re(s)+\vec v$, we have
$\re(s)\,=\, \|s\|\cdot\cos \theta$ and $\vec v=\|s\|\cdot \frac{\vec v}{\|\vec v\|}\sin\theta$,
because of
$\re(s)^2+\|\vec v\|^2\,=\, \|s\|^2\big(\cos^2\theta+\big\|\frac{\vec v}{\|\vec v\|}\big\|^2\sin^2\theta\big)\,=\,\|s\|^2$.
(See {\it e.g.} \cite{Port}.) Since $s$ is a unit quaternion it can be also written as
  \begin{equation}
\label{eq:theta}
s\,=\,e^{\hat n \theta},\quad\text{with}\quad \cos\theta\,=\,\re(s),\end{equation}
 and the unite vector $\hat n=\frac{2(1+x_1)(x_2,x_3,x_4)}{|2(1+x_1)(x_2,x_3,x_4)|}$, where $(x_2,x_3,x_4)$ denotes the vector in $\mathbb R^3$. Moreover, $s^k\,=\, e^{\hat n k\theta}$. Thus
\begin{align*}
\Big|\sum_{k=0}^{2n}(2n-k+1)e^{\hat n k\theta}\Big|^2 \,& =\,\Big|(2n+1)+2ne^{\hat n \theta}+(2n-1)e^{\hat n 2\theta}+\ldots+2e^{\hat n (2n-1)\theta}+e^{\hat n 2n\theta}\Big|^2
\\
&=\,
\sum_{k=0}^{2n}(2n+1-k)^2
\\ & +
2\Big((2n+1)2n+2n(2n-1)+\ldots+3\cdot 2+2\cdot 1\Big)\cos\theta
\\& +
2\Big((2n+1)(2n-1)+2n(2n-2)+\ldots+4\cdot 2+3\cdot 1\Big)\cos(2\theta)
\\& +
\ldots\ldots\ldots\ldots\ldots\ldots\ldots\ldots\ldots\ldots\ldots\ldots
\\& +
2\Big((2n+1)2+2n\cdot1\Big)\cos((2n-1)\theta)
\\& +
2\Big((2n+1)\cdot1\Big)\cos(2n\theta).
\end{align*}
We calculate by using the auxiliary formulas
$$
\alpha_0=\sum_{j=0}^{2n}(2n+1-j)^2=\sum_{j=1}^{2n+1}j^2=\frac{(n+1)(2n+1)(4n+3)}{3},
$$
$$
\alpha_1=2\sum_{j=1}^{2n}j(j+1)=2\frac{(n+1)(2n+1)(4n)}{3},
$$
$$
\alpha_2=2\sum_{j=1}^{2n-1}j(j+2)=2\frac{(n)(2n-1)(4n+5)}{3},
$$
$$
\ldots\ldots\ldots,\quad \alpha_{2n-1}=2(6n+2),\quad\alpha_{2n}=2(2n+1).
$$
In general
\begin{equation}\label{alpha_k}
\alpha_k=\sum_{j=1}^{2n+1-k}j(j+k)=\frac{(2n+1-k)(2n+2-k)(4n+3+k)}{6}.
\end{equation}

Conclude that
$$
\Big|\sum_{k=0}^{2n}(2n-k+1)e^{\hat n k\theta}\Big|^2=\sum_{k=0}^{2n}\alpha_k\cos(k\theta).
$$
Summarizing all that we did, we come to calculation of the following integral
$$
F( {\mathbf e})=c_n\big((2n)!\big)^2\sum_{k=0}^{2n}\alpha_k\int_{\mathbb H^n}dq'\underbrace{\int_{[0,\pi]\times[0,2\pi]}\sin\psi\,d\psi\,d\phi}_{=4\pi}\int_0^{\infty}\frac{r^2\cos(k\theta)}{((1+|q'|^2)^2+r^2)^{4n+6}}dr,
$$
where $r^2=x_2^2+x_3^2+x_4^2$ and $\cos\theta=\frac{(1+|q'|^2)^2-r^2}{(1+|q'|^2)^2+r^2}$.
Recall the formula
$$
\cos(k\theta)=\sum_{l=0}^{k}\left(C^{2l}_{k}\Big(\sum_{m=0}^{l}(-1)^mC^{m}_{l}\cos^{k-2m}\theta\Big)\right),
$$
where $C^s_t=\frac{t!}{s!(t-s)!}$, $s,t\in \{0,1,2,\ldots\}$ and it is vanish otherwise. 
That leads to calculations of
$$
F( {\mathbf e})=c_n4\pi\big((2n)!\big)^2\sum_{k=0}^{2n}\alpha_k\sum_{l=0}^{k}C^{2l}_{k}\sum_{m=0}^{l}(-1)^mC^{m}_{l}\int_{\mathbb H^n}dq'\int_0^{\infty}\frac{r^2\cos^{k-2m}\theta}{((1+|q'|^2)^2+r^2)^{4n+6}}dr.
$$
Substituting the value of $\cos\theta$ and using the notation $d=k-2m$, we concentrate on the calculations of the integrals of type
$$
I_{n,d}(w)=\int_{0}^{\infty}\frac{r^2}{(w^2+r^2)^{4n+6}}\Big(\frac{w^2-r^2}{w^2+r^2}\Big)^{d}dr\quad\text{with}\quad w^2=(1+|q'|^2)^2.
$$

Changing variable $\frac{r^2}{w^2}=t$, we write the integral in the form
$$
I_{n,d}(w)=\frac{(-1)^{d}}{2w^{8n+9}}\int_{0}^{\infty}t^{1/2}(1+t)^{-4n-6-d}(t-1)^d\,dt.
$$
Now we can apply the formula (see \cite{GR})
\begin{equation}\label{integral_beta}
\int_{0}^{\infty}t^{\lambda-1}(1+t)^{-\mu+\nu}(t+\beta)^{-\nu}dt=B(\mu-\lambda,\lambda)_2F_1(\nu,\mu-\lambda,\mu;1-\beta),
\end{equation}
where $\re\mu>\re\lambda>0$,
 with
$$
\beta=-1,\quad \lambda=\frac{3}{2},\quad\nu=-d,\quad\mu=4n+6,\qquad\mu\geq 10>\frac{3}{2}=\lambda>0.
$$
Then
\begin{align}\label{beta}
B(\mu-\lambda,\lambda)&=B(4n+6-3/2,3/2)=\frac{\Gamma(4n+9/2)\Gamma(3/2)}{\Gamma(4n+6)}
\\
&=\frac{(8n+8)!\,\pi}{2^{8n+9}(4n+4)!(4n+5)!},\nonumber
\end{align}
where we used
$$
B(a,b)=\frac{\Gamma(a)\Gamma(b)}{\Gamma(a+b)},\quad\Gamma(1/2)=\sqrt\pi,\quad\Gamma(n)=(n-1)!,\quad\Gamma\Big(\frac{1}{2}+n\Big)=\frac{(2n)!}{4^nn!}\sqrt\pi.
$$
For hypergeometric function
$$
_2F_1(\nu,\mu-\lambda,\mu;1-\beta)= {}_2F_1(-d,4n+9/2,4n+6;2)
$$
we apply the formula
\begin{equation*}
_2F_1(-d,a+1+b+d,a+1;x)=\frac{d!}{(a+1)_d}P^{a,b}_{d}(1-2x),
\end{equation*}
where
$$(a)_d=
\begin{cases}
1\quad & \text{if}\quad  d=0
\\
a(a+1)\ldots(a+d-1) &\text{if} \quad d>0
\end{cases}
$$
 with $a+1=4n+6$ and $b=-d-3/2$, and obtain
$$
_2F_1(-d,4n+9/2,4n+6;2)=\frac{d!}{(4n+6)_d}P_d^{4n+5,-d-3/2}(-3).
$$
To calculate the value of the Jacobi polynomial $P_d^{4n+5,-d-3/2}(-3)$, we use the formula
\begin{equation*}
P^{a,b}_{d}(x)=\sum_{s=0}^{\infty}
\Big(\begin{array}{cc}d+a \\ s\end{array}\Big)
\Big(\begin{array}{cc}d+b \\ d-s\end{array}\Big)
\Big(\frac{x-1}{2}\Big)^{d-s}
\Big(\frac{x+1}{2}\Big)^{s},
\end{equation*}
where $s\geq 0$ and $d-s\geq 0$, and for integer $s$
$$
\Big(\begin{array}{cc}z \\ s\end{array}\Big)=\frac{\Gamma(z+1)}{\Gamma(s+1)\Gamma(z-s+1)}\quad\text{with}\quad\Big(\begin{array}{cc}z \\ s\end{array}\Big)=0\quad\text{for}\quad s<0.
$$
Observe that the terms
$$
\Big(\begin{array}{cc}d+b \\ d-s\end{array}\Big)=\Big(\begin{array}{cc}-3/2 \\ d-s\end{array}\Big)
$$
vanish for $d-s<0$, that allows to conclude that the series in the Jacobi polynomial $P_d^{4n+5,-d-3/2}$ has only finite number of terms and reduces to the sum from $s=0$ to $s=d$. Now we calculate each term in the sum
\begin{equation}\label{Jacobi_final}
P_d^{4n+5,-d-3/2}(-3)=\sum_{s=0}^{s=d}\Big(\begin{array}{cc}4n+5+d \\ s\end{array}\Big)\Big(\begin{array}{cc}-3/2 \\ d-s\end{array}\Big)(-1)^d2^{d-s}.
\end{equation}
We deduce
$$
\Big(\begin{array}{cc}4n+5+d \\ s\end{array}\Big)=\frac{(4n+5+d)!}{s!(4n+5+d-s)!},
$$
$$
\Big(\begin{array}{cc}-3/2 \\ d-s\end{array}\Big)=\frac{\Gamma(-1/2)}{(d-s)!\Gamma(-1/2-d+s)}=\frac{(-1)^{d-s}\big(2(d-s+1)\big)!}{2^{2(d-s)+1}(d-s)!(d-s+1)!},
$$
where we used the formula
$$
\Gamma\Big(\frac{1}{2}-n\Big)=\frac{(-4)^nn!}{(2n)!}\sqrt\pi.
$$
Substituting all terms into~\eqref{Jacobi_final},  we get
$$
P_d^{4n+5,-d-3/2}(-3)=(4n+5+d)!\sum_{s=0}^{s=d}\frac{(-1)^s}{2^{d-s+1}}\frac{\big(2(d-s+1)\big)!}{s!(d-s)!(d-s+1)!(4n+5+d-s)!}.
$$
We finish to calculate the hypergeometric function
\begin{align}\label{hypergeom_final_final}
&_2F_1(-d,4n+9/2,4n+6;2)\nonumber
\\
&=\frac{d!(4n+5+d)!}{(4n+6)_d}\sum_{s=0}^{s=d}\frac{(-1)^s}{2^{d-s+1}}\frac{\big(2(d-s+1)\big)!}{s!(d-s)!(d-s+1)!(4n+5+d-s)!}
\\
&=
(4n+5)!\sum_{s=0}^{s=d}\frac{C_{d}^s}{2^{d-s+1}}\frac{(-1)^s\big(2(d-s+1)\big)!}{(d-s+1)!(4n+5+d-s)!}.\nonumber
\end{align}
Collecting~\eqref{hypergeom_final_final} and~\eqref{beta}, we obtain the value of the integral
\begin{equation}\label{I_final}
I_{n,d}(w)=\frac{(-1)^d\pi}{2^{8n+10}w^{8n+9}}\frac{(8n+8)!}{(4n+4)!}\sum_{s=0}^{s=d}\frac{C_{d}^s}{2^{d-s+1}}\frac{(-1)^s\big(2(d-s+1)\big)!}{(d-s+1)!(4n+5+d-s)!}.
\end{equation}
Replacing $d$ by $k-2m$ and substituting the value of $I_{n,k-2m}(w)$ into $F({\mathbf e})$, we get
\begin{align*}
F({\mathbf e})&=\frac{c_n\pi^2\big((2n)!\big)^2(8n+8)!}{2^{8n+8}(4n+4)!}\sum_{k=0}^{2n}\alpha_k\sum_{l=0}^{k}C^{2l}_{k}\sum_{m=0}^{l}(-1)^{k-m}C^{m}_{l}
\\
&\times
\sum_{s=0}^{k-2m}\frac{C_{k-2m}^s}{2^{k-2m-s+1}}\frac{(-1)^s\big(2(k-2m-s+1)\big)!}{(k-2m-s+1)!(4n+5+k-2m-s)!}\int_{\mathbb H^n}\frac{dq'}{(1+|q'|^2)^{4n+9/2}}
\end{align*}

In order to finish the calculations, we need to evaluate the integral
$$
\int_{\mathbb H^n}\frac{dq'}{(1+|q'|^2)^{4n+9/2}}=\int_{S^{4n-1}}dV^{4n-1}\int_0^\infty\frac{r^{4n-2}\,dr}{(1+r^2)^{4n+9/2}}.
$$
It is well known that the volume of the sphere $S^{4n-1}$ is
$$
\int_{S^{4n-1}}dV^{4n-1}=\frac{\pi^{2n-1/2}}{\Gamma(2n+1/2)}.
$$
Making use the substitution $t=r^2$ we obtain
$$
\int_0^\infty\frac{r^{4n-2}\,dr}{(1+r^2)^{4n+9/2}}=\frac{1}{2}\int_{0}^{\infty}\frac{t^{2n-3/2}\,dt}{(1+t)^{4n+9/2}}=\frac{\Gamma(2n-1/2)\Gamma(2n+5)}{2\Gamma(4n+9/2)}.
$$
Multiplying two latter expressions, we find
\begin{equation}\label{eq:intHn}
\int_{\mathbb H^n}\frac{dq'}{(1+|q'|^2)^{4n+9/2}}=2^{8n+8}\pi^{2n-1}\frac{(2n+4)!}{4n-1}\frac{(4n+4)!}{(8n+8)!}.
\end{equation}
Substituting the value of the integral $ \int_{\mathbb H^n}\frac{dq'}{(1+|q'|^2)^{4n+9/2}}$ to the expression for $F({\mathbf e})$ we get
$$
F({\mathbf e})=c_n\frac{\pi^{2n+1}\big((2n)!\big)^2(2n+4)!}{4n-1}K(n),
$$
where the constant $K(n)$ is given by~\eqref{K}.

Remind that from the other hand $
F({\mathbf e})=\frac{(2n+2)!}{2^{{2n+4}}}
$.
Comparing two expressions for $F({\mathbf e})$ we finish proof of the theorem.
\end{proof}

\begin{cor}
Particularly, the constant in the Cauchy - Szeg\"o kernel for low dimensions are equal to $c_1=\frac{6237}{872\pi^3}$ and $c_2=\frac{11486475}{193472\pi^5}$.
\end{cor}
\begin{proof}
{\sc Case $n=1$.} From one hand $F({\mathbf e})=\frac{3}{8}$. From the other hand we have to calculate the integral
$$
F({\mathbf e})=c_14\pi(2!)^2\left(\int_{\mathbb H^1}\frac{dq'}{(1+|q'|^2)^{4+9/2}}\right)\left(\frac{1}{2}\int_0^{\infty}\frac{t^{1/2}}{(1+t)^{10}}\sum_{k=0}^{2}\alpha_k\cos(k\theta)\,dt\right).
$$
We know that
$$
\int_{\mathbb H^1}\frac{dq'}{(1+|q'|^2)^{4+9/2}}=\frac{2^{16}\pi6!8!}{3\cdot 16!}
$$
by~\eqref{eq:intHn}, and
$$
\sum_{k=0}^{2}\alpha_k\cos(k\theta)=14+8\cos\theta+3\cos(2\theta)=11+8\cos\theta+6\cos^2\theta.
$$
by~\eqref{alpha_k}. Then
\begin{align*}
&\frac{1}{2}\int_0^{\infty}\frac{t^{1/2}}{(1+t)^{10}}\sum_{k=0}^{2}\alpha_k\cos(k\theta)\,dt
\\ & =
\frac{B\left(\frac{17}2,\frac 3 2\right)}{2}\Big(11\, _2F_1(0,17/2,10;2)-8\, _2F_1(-1,17/2,10;2)+6\, _2F_1(-2,17/2,10;2)\Big)
\\
&=\frac{\pi16!}{2^{18}8!9!}(11+8\frac{7}{10}+6\frac{59}{110})
=\frac{\pi16!}{2^{18}8!9!}\frac{218}{11}.
\end{align*}
It gives
$$
F({\mathbf e})=c_14\pi2^2\frac{2^{16}\pi6!8!}{3\cdot 16!}\frac{\pi16!}{2^{18}8!9!}\frac{218}{11}=c_1\pi^3\frac{109}{2079}
$$
and,  finally, $c_1=\frac{6237}{872\pi^3}$

{\sc Case $n=2$.} Again, from one hand we get $F({\mathbf e})=\frac{45}{16}$. From the other hand we get
$$
\int_{\mathbb H^2}\frac{dq'}{(1+|q'|^2)^{8+9/2}}=\frac{2^{24}\pi^38!12!}{7\cdot 24!},
$$
$$
\alpha_0=55,\quad\alpha_1=40,\quad\alpha_2=26,\quad\alpha_3=14,\quad\alpha_4=5,
$$
$$
\cos(2\theta)=2\cos^2\theta-1,\,\cos(3\theta)=4\cos^3\theta-3\cos\theta,\,\cos(4\theta)=8\cos^4\theta-8\cos^2\theta+1,
$$
$$
\sum_{k=0}^{2}\alpha_k\cos(k\theta)=34-2\cos\theta+12\cos^2\theta+56\cos^3\theta+40\cos^4\theta,
$$
\begin{align*}
&\frac{1}{2}\int_0^{\infty}\frac{t^{1/2}}{(1+t)^{14}}\sum_{k=0}^{4}\alpha_k\cos(k\theta)\,dt
\\ & =
\frac{B(25/2,3/2)}{2}\Big(34\, _2F_1(0,25/2,14;2)+2\, _2F_1(-1,25/2,114;2)\\
&\qquad \qquad+12\, _2F_1(-2,25/2,14;2)
-56\, _2F_1(-3,25/2,14;2)+40\, _2F_1(-4,25/2,14;2)\Big)
\\
&=\frac{\pi24!}{2^{26}12!13!}(34-2\frac{11}{14}+12\frac{9}{14}+56\frac{121}{224}+40\frac{1763}{3808})
=\frac{\pi24!}{2^{26}12!13!}\frac{3023}{34}.
\end{align*}
Thus
$$
F({\mathbf e})=c_24\pi(4!)^2\frac{2^{24}\pi^38!12!}{7\cdot 24!}\frac{\pi24!}{2^{26}12!13!}\frac{3023}{34}=c_2\pi^5\frac{12092}{255255},
$$
that leads to $c_2=\frac{11486475}{193472\pi^5}$ and the proof of the corollary is therefore complete.
\end{proof}

\section{Acknowledgment}
The paper was initiated when the authors visited the National Center for
Theoretical Sciences, Hsinchu, Taiwan during
June,  2011. They would like to express their profound gratitude
to the Director of NCTS, Professors Winnie Li for her invitation and
for the warm hospitality extended to them during their stay in
Taiwan.

\end{document}